\documentclass[11pt]{article}
\usepackage[margin=1in]{geometry}
\usepackage[T1]{fontenc}
\usepackage[utf8]{inputenc}
\usepackage[all=normal,bibliography=tight]{savetrees}
\usepackage{tikz}
\usetikzlibrary{decorations.pathmorphing}
\tikzset{snake it/.style={decorate, decoration=snake}}
\usepackage{listings}
\usepackage{cite}
\usepackage{mathrsfs}
\usepackage{amstext,amsfonts,amsthm,amsmath,amssymb}
\usepackage{graphicx,tikz}
\usepackage{comment}
\usepackage{url}
\usepackage{xspace}
\usepackage{todonotes}
\usepackage[shortlabels]{enumitem}
\usepackage[absolute]{textpos}

\def\cqedsymbol{\ifmmode$\lrcorner$\else{\unskip\nobreak\hfil
\penalty50\hskip1em\null\nobreak\hfil$\lrcorner$
\parfillskip=0pt\finalhyphendemerits=0\endgraf}\fi}

\newcommand{\C}{\mathcal{C}}

\newcommand{\dist}{\text{dist}}
\newcommand{\N}{\mathbb{N}}

\newtheorem{lemma}{Lemma}[section]

\newtheorem{corollary}[lemma]{Corollary}
\newtheorem{theorem}[lemma]{Theorem}

\newtheorem{conjecture}[lemma]{Conjecture}
\theoremstyle{definition}

\def\dd{\hbox{-}}

\newcommand*\samethanks[1][\value{footnote}]{\footnotemark[#1]}

\title{Graphs with polynomially many minimal separators}
\author{
Tara Abrishami$^{1}$ \ \ \ \
Maria Chudnovsky$^{1,}$\thanks{Supported by NSF Grant DMS-1763817. This material is based upon work supported by, or in part by, the U.S. Army Research Laboratory and the U. S. Army Research Office under grant number W911NF-16-1-0404.} \ \ \ \
Cemil Dibek$^{1,}$\samethanks \ \ \ \
\\
\\
St\'ephan Thomass\'e$^{2,}$\thanks{Partially supported by the LABEX MILYON (ANR-10-LABX-0070) of Universit\'e de Lyon, within the program Investissements d’Avenir (ANR-11-IDEX-0007) operated by the French National Research Agency (ANR), and by the French National Research Agency under research grant ANR DIGRAPHS ANR-19-CE48-0013-01.} \ \ \ \
Nicolas Trotignon$^{2,}$\samethanks \ \ \ \
Kristina Vu\v{s}kovi\'c$^{3,}$\thanks{Partially supported by EPSRC grant EP/N0196660/1.}\\\\
 $^1${\small Princeton University, Princeton, NJ 08544}\\
 $^2${\small Univ Lyon, EnsL, UCBL, CNRS, LIP, F-69342, LYON Cedex 07, France}\\
\small $^3$School of Computing, University of Leeds, UK
\\
\\
}
\date{}

\begin{document}\maketitle

\begin{abstract}
We show that graphs that do not contain a theta, pyramid, prism, or turtle as an induced subgraph have polynomially many minimal separators. This result is the best possible in the sense that there are graphs with exponentially many minimal separators if only three of the four induced subgraphs are excluded. As a consequence, there is a polynomial time algorithm to solve the maximum weight independent set problem for the class of (theta, pyramid, prism, turtle)-free graphs. Since every prism, theta, and turtle contains an even hole, this also implies a polynomial time algorithm to solve the maximum weight independent set problem for the class of (pyramid, even hole)-free graphs.\\

\end{abstract}

\section{Introduction}\label{sec:intro}
All graphs in this paper are finite and simple. Let $G = (V, E)$ be a graph. A set $C \subseteq V(G)$ is a \emph{minimal separator} of $G$ if there are two distinct connected components $L, R$ of $G \setminus C$ such that $N(L) = N(R) = C$. A class $\mathcal{G}$ of graphs is said to have the \emph{polynomial separator property} if there exists a constant $c$ such that every graph $G \in \mathcal{G}$ has at most $|V(G)|^{c}$ minimal separators. 

The polynomial separator property has proven to be a desirable property due to its connection with potential maximal cliques and the maximum weight treewidth $k$ induced subgraph problem. Given a graph $G$, a nonnegative weight function on $V(G)$, and an integer $k$, the \textsc{maximum weight treewidth $k$ induced subgraph} problem (\textsc{MWT$k$ISg}) asks for a maximum-weight induced subgraph of $G$ of treewidth less than $k$. The \textsc{maximum weight independent set} problem (\textsc{MWIS}), which asks for an independent set of $G$ with maximum weight, and the \textsc{feedback vertex set} problem (\textsc{FVS}), which asks for a minimum-size set $X \subseteq V(G)$ such that $G \setminus X$ is a forest, are special cases of \textsc{MWT$k$ISg} when $k = 1$ and $k = 2$, respectively. Recently, significant progress was made regarding the complexity of \textsc{MWIS} in various graph classes using potential maximal cliques, originally developed by Bouchitté and Todinca \cite{BouTod0, BouTod}. A milestone result with this approach was obtained in 2014 by Lokshtanov, Vatshelle, and Villanger \cite{LVV}, who designed a polynomial-time algorithm for \textsc{MWIS} in $P_5$-free graphs. Later, using the same framework, Grzesik et al. \cite{GKPP} provided a polynomial-time algorithm for \textsc{MWIS} in $P_6$-free graphs. More recently, Abrishami et al. \cite{ACPRS} extended the framework of potential maximal cliques to \textsc{MWT$k$ISg}, and gave a polynomial-time algorithm for \textsc{MWIS} in graphs with no induced cycle of length five or greater, and for \textsc{FVS} in $P_5$-free graphs.

Given an integer $k$, it is known that \textsc{MWT$k$ISg} can be solved in polynomial time for graphs that have polynomially many potential maximal cliques. Minimal separators are closely related to potential maximal cliques: it was shown in \cite{BouTod} that a graph has polynomially many potential maximal cliques if and only if it has polynomially many minimal separators. Consequently, \textsc{MWT$k$ISg} is polynomial-time solvable in any class of graphs that has the polynomial separator property. It is therefore interesting to find classes of graphs where the number of minimal separators is bounded by a polynomial. We now define four graphs of interest to us (see also Figure \ref{fig:forbidden_isgs}):
\begin{itemize}
\itemsep-0.2em
\item A \emph{theta} is a graph $G$ consisting of two nonadjacent vertices $a, b$ and three paths $P_1, P_2, P_3$, each from $a$ to $b$, and otherwise vertex-disjoint, such that for $1 \leq i < j \leq 3$, $V(P_i) \cup V(P_j)$ induces a hole in $G$. In particular, each of $P_1, P_2, P_3$ has at least two edges. We say that $G$ is a theta between $a$ and $b$.

\item A \emph{pyramid} is a graph $G$ consisting of a vertex $a$ and a triangle $\{b_1, b_2, b_3\}$, and three paths $P_1, P_2, P_3$, such that: $P_i$ is between $a$ and $b_i$ for $i = 1, 2, 3$; for $1 \leq i < j \leq 3$, $P_i, P_j$ are vertex-disjoint except for $a$ and $V(P_i) \cup V(P_j)$ induces a hole in $G$; and in particular at most one of $P_1, P_2, P_3$ has only one edge. We say that $G$ is a pyramid from $a$ to $b_1b_2b_3$.
    
\item A \emph{prism} is a graph $G$ consisting of two vertex-disjoint triangles $\{a_1, a_2, a_3\}$, $\{b_1, b_2, b_3\}$, and three paths $P_1, P_2, P_3$, pairwise vertex-disjoint, where each $P_i$ has ends $a_i, b_i$, and for $1 \leq i < j \leq 3$, $V(P_i) \cup V(P_j)$ induces a hole in $G$. In particular, each of $P_1, P_2, P_3$ has at least one edge. We say $G$ is a prism between $a_1a_2a_3$ and $b_1b_2b_3$.

\item A \emph{turtle} is a graph $G$ consisting of two vertex-disjoint paths $P_1, P_2$ and two adjacent vertices $x, y \in V(G) \setminus (V(P_1) \cup V(P_2))$ such that for $i=1,2$, $P_i$ is from $a_i$ to $b_i$, $a_1$ is adjacent to $a_2$, $b_1$ is adjacent to $b_2$, $V(P_1) \cup V(P_2)$ induces a hole in $G$, $x$ has at least three neighbors in $P_1$ and no neighbors in $P_2$, and $y$ has at least three neighbors in $P_2$ and no neighbors in $P_1$. We say that $G$ is an \emph{$xy$-turtle} where we call $x$ and $y$ the \emph{centers} of $G$.

\end{itemize}

\vspace{-0.4cm}

\begin{figure}[ht]
\begin{center}
\begin{tikzpicture}[scale=0.29]

\node[inner sep=2.5pt, fill=black, circle] at (0,0)(v1){}; 
\node[inner sep=2.5pt, fill=black, circle] at (3, 3)(v2){};
\node[inner sep=2.5pt, fill=black, circle] at (3, 0)(v3){}; 
\node[inner sep=2.5pt, fill=black, circle] at (3, -3)(v4){}; 
\node[inner sep=2.5pt, fill=black, circle] at (6, 0)(v5){};

\node[inner sep=2.5pt, fill=white, circle] at (0, -4.8)(v21){};

\draw[black, thick] (v1) -- (v2);
\draw[black, thick] (v1) -- (v3);
\draw[black, thick] (v1) -- (v4);
\draw[black, dotted, thick] (v2) -- (v5);
\draw[black, dotted, thick] (v3) -- (v5);
\draw[black, dotted, thick] (v4) -- (v5);

\end{tikzpicture}
\hspace{0.7cm}
\begin{tikzpicture}[scale=0.29]

\node[inner sep=2.5pt, fill=black, circle] at (0,0)(v1){}; 
\node[inner sep=2.5pt, fill=black, circle] at (3, 3)(v2){}; 
\node[inner sep=2.5pt, fill=black, circle] at (3, -3)(v3){}; 
\node[inner sep=2.5pt, fill=black, circle] at (6, 0)(v4){};
\node[inner sep=2.5pt, fill=black, circle] at (9, 3)(v5){};
\node[inner sep=2.5pt, fill=black, circle] at (9, -3)(v6){}; 

\node[inner sep=2.5pt, fill=white, circle] at (0, -4.8)(v21){}; 

\draw[black, thick] (v1) -- (v2);
\draw[black, thick] (v1) -- (v3);
\draw[black, dotted, thick] (v1) -- (v4);
\draw[black, dotted, thick] (v2) -- (v5);
\draw[black, dotted, thick] (v3) -- (v6);
\draw[black, thick] (v4) -- (v5);
\draw[black, thick] (v4) -- (v6);
\draw[black, thick] (v5) -- (v6);

\end{tikzpicture}
\hspace{0.7cm}
\begin{tikzpicture}[scale=0.29]

\node[inner sep=2.5pt, fill=black, circle] at (-3, 3)(v1){}; 
\node[inner sep=2.5pt, fill=black, circle] at (0, 0)(v2){}; 
\node[inner sep=2.5pt, fill=black, circle] at (-3, -3)(v3){}; 
\node[inner sep=2.5pt, fill=black, circle] at (6, 3)(v4){}; 
\node[inner sep=2.5pt, fill=black, circle] at (3, 0)(v5){}; 
\node[inner sep=2.5pt, fill=black, circle] at (6, -3)(v6){};

\node[inner sep=2.5pt, fill=white, circle] at (0, -4.8)(v21){}; 

\draw[black, thick] (v1) -- (v2);
\draw[black, thick] (v1) -- (v3);
\draw[black, thick] (v2) -- (v3);
\draw[black, dotted, thick] (v1) -- (v4);
\draw[black, dotted, thick] (v2) -- (v5);
\draw[black, dotted, thick] (v3) -- (v6);
\draw[black, thick] (v4) -- (v5);
\draw[black, thick] (v4) -- (v6);
\draw[black, thick] (v5) -- (v6);

\end{tikzpicture}
\hspace{0.7cm}
\begin{tikzpicture}[scale=0.30]

\node[label=above:{$\scriptstyle{a_1}$}, inner sep=2.5pt, fill=black, circle] at (-0.7, 3.5)(v1){}; 
\node[label=below:{$\scriptstyle{b_1}$}, inner sep=2.5pt, fill=black, circle] at (-0.7, -3.5)(v2){};
\node[label=above:{$\scriptstyle{a_2}$}, inner sep=2.5pt, fill=black, circle] at (0.7, 3.5)(v1){}; 
\node[label=below:{$\scriptstyle{b_2}$}, inner sep=2.5pt, fill=black, circle] at (0.7, -3.5)(v2){};
\node[label=below:{$\scriptstyle{x}$}, inner sep=2.5pt, fill=black, circle] at (-1.2, 0)(v3){}; 
\node[label=below:{$\scriptstyle{y}$}, inner sep=2.5pt, fill=black, circle] at (1.2, 0)(v4){};
\node[inner sep=2.5pt, fill=black, circle] at (-3.6, 0)(v5){}; 
\node[inner sep=2.5pt, fill=black, circle] at (-3.3, 1.5)(v6){};
\node[inner sep=2.5pt, fill=black, circle] at (-3.3, -1.5)(v7){}; 
\node[inner sep=2.5pt, fill=black, circle] at (3.6, 0)(v8){}; 
\node[inner sep=2.5pt, fill=black, circle] at (3.3, 1.5)(v9){}; 
\node[inner sep=2.5pt, fill=black, circle] at (3.3, -1.5)(v10){}; 

\draw[black, thick] (v3) -- (v4);
\draw[black, thick] (v3) -- (v5);
\draw[black, thick] (v3) -- (v6);
\draw[black, thick] (v3) -- (v7);
\draw[black, thick] (v4) -- (v8);
\draw[black, thick] (v4) -- (v9);
\draw[black, thick] (v4) -- (v10);

\draw[black, thick] (0,0) circle (3.6cm);

\end{tikzpicture}
\end{center}
\vspace{-0.75cm}
\caption{Theta, pyramid, prism, and turtle}
\label{fig:forbidden_isgs}
\end{figure}
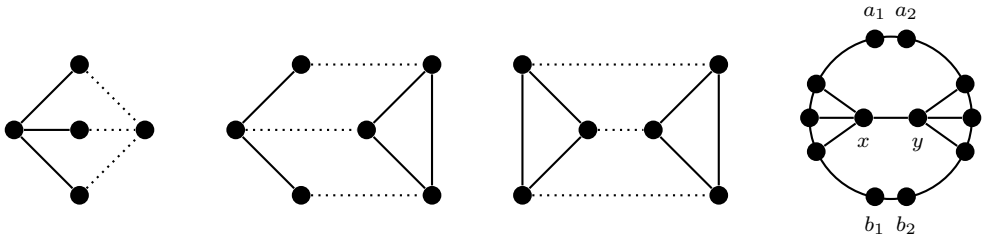

Thetas, pyramids, prisms, and turtles are interesting because they provide examples of graphs with exponentially many minimal separators. Specifically, we have the following examples of graphs with exponentially many minimal separators (see also Figure \ref{fig:k_forbidden}).
\begin{itemize}
\itemsep-0.2em
\item A \emph{$k$-theta} is a graph $G$ with vertex set $V(G) = \{a, a_1, \dots, a_k, b, b_1, \dots, b_k\}$, and its set of edges consists of the pairs of the following form: $aa_i$, $bb_i$, and $a_ib_i$ for $1 \leq i \leq k$.

\item A \emph{$k$-pyramid} is a graph $G$ with vertex set $V(G) = \{a, a_1, \dots, a_k, b_1, \dots, b_k\}$, and its set of edges consists of the pairs of the following form: $aa_i$ and $a_ib_i$ for $1 \leq i \leq k$, and $b_ib_j$ for $1 \leq i < j \leq k$.

\item A \emph{$k$-prism} is a graph $G$ consisting of two cliques of size $k$ and a $k$-edge matching between them. More precisely, $V(G)=\{a_1, \dots, a_k, b_1, \dots, b_k\}$, each of the sets $\{a_1, \dots, a_k\}$ and $\{b_1, \dots, b_k\}$ is a clique, and $a_ib_i \in E(G)$ for $1 \leq i \leq k$, and there are no other edges in $G$.

\item A \emph{$k$-turtle} is a graph $G$ with two non-adjacent vertices $a, b \in V(G)$, two paths $P_1$ and $P_2$ from $a$ to $b$, vertex-disjoint except for $a$ and $b$, such that $V(P_1) \cup V(P_2)$ induces a hole $H$ in $G$. Also, for $1 \leq i \leq k$, $x_i, y_i \in V(G) \setminus V(H)$ such that $x_iy_i \in E(G)$, and $x_i$ has at least three neighbors in $P_1$ and no neighbors in $P_2$, and $y_i$ has at least three neighbors in $P_2$ and no neighbors in $P_1$. Furthermore, the neighbors of $x_i$'s in $P_1$ and the neighbors of $y_i$'s in $P_2$ are nested along $P_1$ and $P_2$ as shown in Figure \ref{fig:k_forbidden}. Specifically, the neighbors of $x_i$ in $P_1$ are between $a$ and the neighbors of $x_j$ in $P_1$ for all $1 \leq i < j \leq k$ and the neighbors of $y_i$ in $P_2$ are between $a$ and the neighbors of $y_j$ in $P_2$ for all $1 \leq i < j \leq k$.
\end{itemize}

\vspace{-0.4cm}

\begin{figure}[ht]
\begin{center}
\begin{tikzpicture}[scale=0.33]

\node[inner sep=2.5pt, fill=black, circle] at (0,0)(v1){}; 
\node[inner sep=2.5pt, fill=black, circle] at (2, 4)(v2){};
\node[inner sep=2.5pt, fill=black, circle] at (2, 2)(v3){}; 
\node[inner sep=2.5pt, fill=black, circle] at (2, -2)(v4){}; 
\node[inner sep=2.5pt, fill=black, circle] at (2, -4)(v5){};
\node[inner sep=2.5pt, fill=black, circle] at (7, 0)(v6){}; 
\node[inner sep=2.5pt, fill=black, circle] at (5, 4)(v7){};
\node[inner sep=2.5pt, fill=black, circle] at (5, 2)(v8){};
\node[inner sep=2.5pt, fill=black, circle] at (5, -2)(v9){};
\node[inner sep=2.5pt, fill=black, circle] at (5, -4)(v10){};

\node[inner sep=2.5pt, fill=white, circle] at (0, -6.5)(v21){};

\draw[black, thick] (v1) -- (v2);
\draw[black, thick] (v1) -- (v3);
\draw[black, thick] (v1) -- (v4);
\draw[black, thick] (v1) -- (v5);
\draw[black, thick] (v6) -- (v7);
\draw[black, thick] (v6) -- (v8);
\draw[black, thick] (v6) -- (v9);
\draw[black, thick] (v6) -- (v10);
\draw[black, thick] (v2) -- (v7);
\draw[black, thick] (v3) -- (v8);
\draw[black, thick] (v4) -- (v9);
\draw[black, thick] (v5) -- (v10);

\node at (3.5,0.5) {\vdots};

\end{tikzpicture}
\hspace{0.4cm}
\begin{tikzpicture}[scale=0.33]

\node[inner sep=2.5pt, fill=black, circle] at (0,0)(v1){}; 
\node[inner sep=2.5pt, fill=black, circle] at (2, 4)(v2){};
\node[inner sep=2.5pt, fill=black, circle] at (2, 2)(v3){}; 
\node[inner sep=2.5pt, fill=black, circle] at (2, -2)(v4){}; 
\node[inner sep=2.5pt, fill=black, circle] at (2, -4)(v5){};
\node[inner sep=2.5pt, fill=black, circle] at (6, 4)(v7){};
\node[inner sep=2.5pt, fill=black, circle] at (6, 2)(v8){};
\node[inner sep=2.5pt, fill=black, circle] at (6, -2)(v9){};
\node[inner sep=2.5pt, fill=black, circle] at (6, -4)(v10){};

\node[inner sep=2.5pt, fill=white, circle] at (0, -6.5)(v21){};

\draw[black, thick] (v1) -- (v2);
\draw[black, thick] (v1) -- (v3);
\draw[black, thick] (v1) -- (v4);
\draw[black, thick] (v1) -- (v5);
\draw[black, thick] (v2) -- (v7);
\draw[black, thick] (v3) -- (v8);
\draw[black, thick] (v4) -- (v9);
\draw[black, thick] (v5) -- (v10);

\draw[black, thick] (v7) -- (v8);
\draw[black, thick] (v8) -- (v9);
\draw[black, thick] (v9) -- (v10);

\draw[black, thick](v7) edge [bend left=30] (v9);
\draw[black, thick](v7) edge [bend left=45] (v10);
\draw[black, thick](v8) edge [bend left=30] (v10);

\node at (4,0.5) {\vdots};

\end{tikzpicture}
\hspace{0.23cm}
\begin{tikzpicture}[scale=0.33]

\node[inner sep=2.5pt, fill=black, circle] at (2, 4)(v2){};
\node[inner sep=2.5pt, fill=black, circle] at (2, 2)(v3){}; 
\node[inner sep=2.5pt, fill=black, circle] at (2, -2)(v4){}; 
\node[inner sep=2.5pt, fill=black, circle] at (2, -4)(v5){};
\node[inner sep=2.5pt, fill=black, circle] at (6, 4)(v7){};
\node[inner sep=2.5pt, fill=black, circle] at (6, 2)(v8){};
\node[inner sep=2.5pt, fill=black, circle] at (6, -2)(v9){};
\node[inner sep=2.5pt, fill=black, circle] at (6, -4)(v10){};

\node[inner sep=2.5pt, fill=white, circle] at (0, -6.5)(v21){};

\draw[black, thick] (v2) -- (v7);
\draw[black, thick] (v3) -- (v8);
\draw[black, thick] (v4) -- (v9);
\draw[black, thick] (v5) -- (v10);

\draw[black, thick] (v2) -- (v3);
\draw[black, thick] (v3) -- (v4);
\draw[black, thick] (v4) -- (v5);

\draw[black, thick](v2) edge [bend right=30] (v4);
\draw[black, thick](v2) edge [bend right=45] (v5);
\draw[black, thick](v3) edge [bend right=30] (v5);

\draw[black, thick] (v7) -- (v8);
\draw[black, thick] (v8) -- (v9);
\draw[black, thick] (v9) -- (v10);

\draw[black, thick](v7) edge [bend left=30] (v9);
\draw[black, thick](v7) edge [bend left=45] (v10);
\draw[black, thick](v8) edge [bend left=30] (v10);

\node at (4,0.5) {\vdots};

\end{tikzpicture}
\hspace{0.23cm}
\begin{tikzpicture}[scale=0.32]

\node[label=above:{$\scriptstyle{a}$}, inner sep=2.5pt, fill=black, circle] at (0, 6)(v1){}; 
\node[label=below:{$\scriptstyle{b}$}, inner sep=2.5pt, fill=black, circle] at (0, -6)(v2){};
\node[inner sep=2.5pt, fill=black, circle] at (-1, 4)(v3){}; 
\node[inner sep=2.5pt, fill=black, circle] at (1, 4)(v4){};
\node[inner sep=2.5pt, fill=black, circle] at (-1, 2)(v5){}; 
\node[inner sep=2.5pt, fill=black, circle] at (1, 2)(v6){};
\node[inner sep=2.5pt, fill=black, circle] at (-1, -2)(v7){}; 
\node[inner sep=2.5pt, fill=black, circle] at (1, -2)(v8){};
\node[inner sep=2.5pt, fill=black, circle] at (-1, -4)(v9){}; 
\node[inner sep=2.5pt, fill=black, circle] at (1, -4)(v10){};

\node[inner sep=2.5pt, fill=black, circle] at (-2, 5)(v11){}; 
\node[inner sep=2.5pt, fill=black, circle] at (-2.5, 4.25)(v12){};
\node[inner sep=2.5pt, fill=black, circle] at (-3, 3.5)(v13){}; 
\node[inner sep=2.5pt, fill=black, circle] at (2, 5)(v14){};
\node[inner sep=2.5pt, fill=black, circle] at (2.5, 4.25)(v15){}; 
\node[inner sep=2.5pt, fill=black, circle] at (3, 3.5)(v16){};

\node[inner sep=2.5pt, fill=black, circle] at (-3.25, 2.5)(v17){}; 
\node[inner sep=2.5pt, fill=black, circle] at (-3.5, 1.75)(v18){};
\node[inner sep=2.5pt, fill=black, circle] at (-3.65, 1)(v19){}; 
\node[inner sep=2.5pt, fill=black, circle] at (3.25, 2.5)(v20){};
\node[inner sep=2.5pt, fill=black, circle] at (3.5, 1.75)(v21){}; 
\node[inner sep=2.5pt, fill=black, circle] at (3.65, 1)(v22){};

\node[inner sep=2.5pt, fill=black, circle] at (-3.25, -2.5)(v23){}; 
\node[inner sep=2.5pt, fill=black, circle] at (-3.5, -1.75)(v24){};
\node[inner sep=2.5pt, fill=black, circle] at (-3.65, -1)(v25){}; 
\node[inner sep=2.5pt, fill=black, circle] at (3.25, -2.5)(v26){};
\node[inner sep=2.5pt, fill=black, circle] at (3.5, -1.75)(v27){}; 
\node[inner sep=2.5pt, fill=black, circle] at (3.65, -1)(v28){};

\node[inner sep=2.5pt, fill=black, circle] at (-2, -5)(v29){}; 
\node[inner sep=2.5pt, fill=black, circle] at (-2.5, -4.25)(v30){};
\node[inner sep=2.5pt, fill=black, circle] at (-3, -3.5)(v31){}; 
\node[inner sep=2.5pt, fill=black, circle] at (2, -5)(v32){};
\node[inner sep=2.5pt, fill=black, circle] at (2.5, -4.25)(v33){}; 
\node[inner sep=2.5pt, fill=black, circle] at (3, -3.5)(v34){};

\draw[black, thick](v1) edge [bend right=75] (v2);
\draw[black, thick](v1) edge [bend left=75] (v2);

\draw[black, thick] (v3) -- (v4);
\draw[black, thick] (v5) -- (v6);
\draw[black, thick] (v7) -- (v8);
\draw[black, thick] (v9) -- (v10);

\draw[black, thick] (v3) -- (v11);
\draw[black, thick] (v3) -- (v12);
\draw[black, thick] (v3) -- (v13);
\draw[black, thick] (v4) -- (v14);
\draw[black, thick] (v4) -- (v15);
\draw[black, thick] (v4) -- (v16);

\draw[black, thick] (v5) -- (v17);
\draw[black, thick] (v5) -- (v18);
\draw[black, thick] (v5) -- (v19);
\draw[black, thick] (v6) -- (v20);
\draw[black, thick] (v6) -- (v21);
\draw[black, thick] (v6) -- (v22);

\draw[black, thick] (v7) -- (v23);
\draw[black, thick] (v7) -- (v24);
\draw[black, thick] (v7) -- (v25);
\draw[black, thick] (v8) -- (v26);
\draw[black, thick] (v8) -- (v27);
\draw[black, thick] (v8) -- (v28);

\draw[black, thick] (v9) -- (v29);
\draw[black, thick] (v9) -- (v30);
\draw[black, thick] (v9) -- (v31);
\draw[black, thick] (v10) -- (v32);
\draw[black, thick] (v10) -- (v33);
\draw[black, thick] (v10) -- (v34);

\node at (0,0.5) {\vdots};

\end{tikzpicture}
\hspace{0.42cm}
\begin{tikzpicture}[scale=0.3]

\node[inner sep=2.5pt, fill=black, circle] at (0, 0)(v1){};
\node[inner sep=2.5pt, fill=black, circle] at (2.5, 0)(v2){};
\node[inner sep=2.5pt, fill=black, circle] at (6, 0)(v3){};

\node[inner sep=2.5pt, fill=black, circle] at (0, 2)(v4){}; 
\node[inner sep=2.5pt, fill=black, circle] at (2.5, 2)(v5){};
\node[inner sep=2.5pt, fill=black, circle] at (4.5, 2)(v6){};
\node[inner sep=2.5pt, fill=black, circle] at (6, 1.2)(v7){};
\node[inner sep=2.5pt, fill=black, circle] at (6, 2.8)(v8){};

\node[inner sep=2.5pt, fill=black, circle] at (0, 3.5)(v9){};
\node[inner sep=2.5pt, fill=black, circle] at (0, 5)(v10){};
\node[inner sep=2.5pt, fill=black, circle] at (2.5, 5)(v11){};
\node[inner sep=2.5pt, fill=black, circle] at (4.5, 5)(v12){};
\node[inner sep=2.5pt, fill=black, circle] at (6, 4.2)(v13){}; 
\node[inner sep=2.5pt, fill=black, circle] at (6, 5.8)(v14){}; 

\node[inner sep=2.5pt, fill=black, circle] at (0, 9)(v15){};
\node[inner sep=2.5pt, fill=black, circle] at (2.5, 9)(v16){};
\node[inner sep=2.5pt, fill=black, circle] at (4.5, 9)(v17){};
\node[inner sep=2.5pt, fill=black, circle] at (6, 8.2)(v18){};
\node[inner sep=2.5pt, fill=black, circle] at (6, 9.8)(v19){};

\node[inner sep=2.5pt, fill=black, circle] at (0, 10.5)(v20){};
\node[inner sep=2.5pt, fill=black, circle] at (0, 12)(v21){};
\node[inner sep=2.5pt, fill=black, circle] at (2.5, 12)(v22){};
\node[inner sep=2.5pt, fill=black, circle] at (4.5, 12)(v23){}; 
\node[inner sep=2.5pt, fill=black, circle] at (6, 11.2)(v24){}; 
\node[inner sep=2.5pt, fill=black, circle] at (6, 12.8)(v25){};

\node[inner sep=2.5pt, fill=black, circle] at (0, 14)(v26){};
\node[inner sep=2.5pt, fill=black, circle] at (2.5, 14)(v27){};
\node[inner sep=2.5pt, fill=black, circle] at (6, 14)(v28){};

\node[inner sep=2.5pt, fill=white, circle] at (1, -0.7)(v30){};

\draw[black, thick] (v1) -- (v2);
\draw[black, thick] (v2) -- (v3);
\draw[black, thick] (v1) -- (v4);
\draw[black, thick] (v4) -- (v5);
\draw[black, thick] (v5) -- (v6);
\draw[black, thick] (v6) -- (v7);
\draw[black, thick] (v6) -- (v8);
\draw[black, thick] (v7) -- (v8);
\draw[black, thick] (v3) -- (v7);
\draw[black, thick] (v4) -- (v9);
\draw[black, thick] (v9) -- (v10);
\draw[black, thick] (v10) -- (v11);
\draw[black, thick] (v11) -- (v12);
\draw[black, thick] (v12) -- (v13);
\draw[black, thick] (v12) -- (v14);
\draw[black, thick] (v13) -- (v14);
\draw[black, thick] (v8) -- (v13);
\draw[black, thick] (v10) -- (v15);
\draw[black, thick] (v15) -- (v16);
\draw[black, thick] (v16) -- (v17);
\draw[black, thick] (v17) -- (v18);
\draw[black, thick] (v17) -- (v19);
\draw[black, thick] (v18) -- (v19);
\draw[black, thick] (v14) -- (v18);
\draw[black, thick] (v15) -- (v20);
\draw[black, thick] (v20) -- (v21);
\draw[black, thick] (v21) -- (v22);
\draw[black, thick] (v22) -- (v23);
\draw[black, thick] (v23) -- (v24);
\draw[black, thick] (v23) -- (v25);
\draw[black, thick] (v24) -- (v25);
\draw[black, thick] (v19) -- (v24);
\draw[black, thick] (v21) -- (v26);
\draw[black, thick] (v26) -- (v27);
\draw[black, thick] (v27) -- (v28);
\draw[black, thick] (v25) -- (v28);

\node at (2.5,7.5) {\vdots};

\end{tikzpicture}
\end{center}
\vspace{-0.6cm}
\caption{$k$-theta, $k$-pyramid, $k$-prism, $k$-turtle, and $k$-ladder}
\label{fig:k_forbidden}
\end{figure}
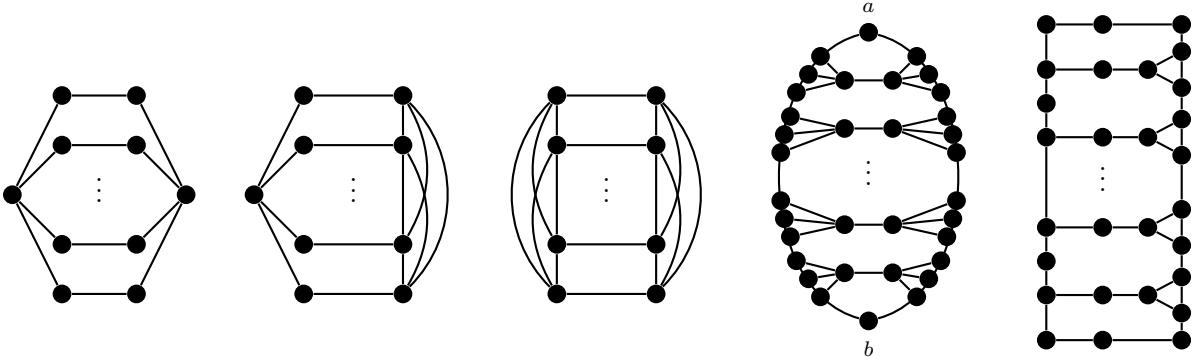

There are other examples of graphs with exponentially many minimal separators, such as the $k$-ladder shown in Figure \ref{fig:k_forbidden}. The $k$-ladder also contains a pyramid. In view of these examples, it is natural to ask whether excluding theta, pyramid, prism, and turtle in a graph is enough to obtain a polynomial number of minimal separators. This was conjectured in \cite{CTTV}:

\begin{conjecture}[\hspace{1sp}\cite{CTTV}]
There is a polynomial $P$ such that every graph $G$ that contains no theta, pyramid, prism, or turtle has at most $P(|V(G)|)$ minimal separators.
\label{conj:main_conj}
\end{conjecture}

Here we prove Conjecture \ref{conj:main_conj}. Let $\C$ be the class of (theta, pyramid, prism, turtle)-free graphs. We prove that the graphs in $\C$ have polynomially many minimal separators. Note that in view of the results in \cite{BBC}, listing the minimal separators of a graph can be done in polynomial time in the size of the graph and the number of its minimal separators. Our proof that the graphs in $\C$ have polynomially many minimal separators is algorithmic in nature, so we include a polynomial-time algorithm here to construct minimal separators of graphs in $\C$ for completeness.
\begin{theorem}
Let $G \in \C$. One can construct a set $\mathcal{S}$ of size at most $|V(G)|^{18}$ in polynomial time such that $\mathcal{S}$ is the set of all minimal separators of $G$.
\label{thm:main_thm}
\end{theorem}

Since the graphs in Figure \ref{fig:k_forbidden} have exponentially many minimal separators, Theorem \ref{thm:main_thm} is in a sense the best possible. Moreover, as explained above, given an integer $k$, Theorem \ref{thm:main_thm} implies that \textsc{MWT$k$ISg} can be solved in polynomial time for graphs in $\C$. To be more precise, let $n, m, p, s$ denote, respectively, the number of vertices, the number of edges, the number of potential maximal cliques, and the number of minimal separators of a graph $G$. It is proved in \cite{BBC} that computing the minimal separators of $G$ can be done in time $\mathcal{O}(n^3 s)$. In \cite{BouTod}, it is proved that $p \leq \mathcal{O}(ns^2 + ns + 1)$ and that the potential maximal cliques of $G$ can be listed in time $\mathcal{O}(n^2ms^2)$. In \cite{ACPRS}, it is proved that given the list of potential maximal cliques of $G$ and an integer $k$, if $p$ is polynomial in $n$, then \textsc{MWT$k$ISg} can be solved in time $n^{\mathcal{O}(k)}$. By Theorem \ref{thm:main_thm}, for a graph $G \in \C$, we have $s \leq \mathcal{O}(n^{18})$, and so $p \leq \mathcal{O}(n^{37})$. Therefore, \textsc{MWT$k$ISg} can be solved in time $n^{O(k)}$ in $\C$. Using results from \cite{LVV}, a better complexity for \textsc{MWIS} can be achieved. In \cite{LVV}, based on \cite{FomVil}, it is proved that, given the list of potential maximal cliques, \textsc{MWIS} can be solved in time $\mathcal{O}(n^5mp)$ in any graph. Therefore, \textsc{MWIS} can be solved in time $\mathcal{O}(n^{44})$ in $\C$.

It is easy to observe that every prism, theta, and turtle contains an even hole. Therefore, the following is an immediate corollary of Theorem \ref{thm:main_thm}.

\begin{corollary}
The class of (pyramid, even hole)-free graphs has the polynomial separator property.
\label{cor:evenhole_pyramid}
\end{corollary}

\vspace{-0.5cm}

In \cite{CTTV}, a better bound than the one given in Theorem \ref{thm:main_thm} is achieved for (pyramid, even hole)-free graphs. In particular, Corollary \ref{cor:evenhole_pyramid} implies that \textsc{MWT$k$ISg} and \textsc{MWIS} can be solved in (pyramid, even hole)-free graphs in polynomial time. A \emph{cap} is a cycle of length at least five with exactly one chord and that chord creates a triangle with the cycle. Since every pyramid contains a cap, Corollary \ref{cor:evenhole_pyramid} generalizes a result of \cite{CGP} where it is shown that \textsc{MWIS} can be solved in (cap, even hole)-free graphs in polynomial time.

We conjecture a stronger version of Theorem \ref{thm:main_thm}. For an integer $k \geq 3$, a graph $G$ is called a \emph{$k$-creature} if it is given as follows: $V(G) = A \cup B \cup \{x_1, \dots, x_k\} \cup \{y_1, \dots, y_k\}$ such that
\begin{enumerate}[(i)]
\itemsep-0.2em
\item $G[A]$ and $G[B]$ are connected, and $A$ is anticomplete to $B$,
\item for $i=1,\dots,k$, $x_iy_i \in E(G)$, $x_i$ has a neighbor in $A$ and is anticomplete to $B$, $y_i$ has a neighbor in $B$ and is anticomplete to $A$, and
\item for $1 \leq i, j \leq k$ with $i\neq j$, $x_iy_j \notin E(G)$.
\end{enumerate}

We observe that if $G$ is a $k$-creature, then $G$ contains a theta, pyramid, prism, or turtle; see Lemma \ref{lemma:three_paths_lemma} for details. We conjecture the following:

\begin{conjecture}
There exists $f:\N \to \N$ such that if no induced subgraph of $G$ is a $k$-creature, then $G$ has at most $|V(G)|^{f(k)}$ minimal separators.
\label{conj:creature_conj}
\end{conjecture}

Observe that even if Conjecture \ref{conj:creature_conj} is true, it does not provide a full characterization of classes with the polynomial separator property. For example, for every integer $k\geq 1$, let $T_k$ be a $k$-turtle such that the two paths $P_1$ and $P_2$ both have length $2^{2^k}$. Let $\mathcal{D}$ be the class of graphs formed by all induced subgraphs of the graphs $T_k$, $k \geq 1$. Observe that $T_k$ has $2^k$ minimal separators, which is polynomial in $|V(T_k)|$ since $|V(T_k)| \geq 2^{2^k}$, and so $\mathcal{D}$ has the polynomial separator property. However, $\mathcal{D}$ contains $k$-creatures with $k$ arbitrarily large.

We prove a weaker version of Conjecture \ref{conj:creature_conj}. The proof can also be found in \cite{CTTV}. We say that a graph $G$ is an \emph{immature $k$-creature} if $V(G)$ can be partitioned into two sets $X = \{x_1, \dots, x_k\}$ and $Y = \{y_1, \dots, y_k\}$ such that the only edges between $X$ and $Y$ are the edges $x_iy_i$ for $i=1,\dots,k$. The edges among vertices of $X$ and vertices of $Y$ are unrestricted.

\begin{theorem}
Let $k \geq 1$ be an integer and let $G$ be a graph on $n$ vertices such that no induced subgraph of $G$ is an immature $k$-creature. Then, $G$ has at most $\mathcal{O}(n^{2k-2})$ minimal separators that can be enumerated in time $\mathcal{O}(n^{2k})$.
\label{thm:immature_k_creature}
\end{theorem}
\begin{proof}
Let $a$ and $b$ be two non-adjacent vertices in $G$. Let $C$ be a minimal separator that separates $a$ and $b$. Let $A$ and $B$ be the components of $G \setminus C$ that contain $a$ and $b$, respectively. By the minimality of $C$, every vertex in $C$ has a neighbor in $A$. It is therefore well-defined to consider an inclusion-wise minimal subset $X_A$ of $A$ such that $C \subseteq N(X_A)$. For every $x \in X_A$, there exists a vertex $c \in C$ such that $xc \in E(G)$ and no other vertex of $X_A$ is adjacent to $c$, for otherwise, $X_A \setminus \{x\}$ would contradict the minimality of $X_A$. It follows that $G[X_A \cup C]$ contains an immature $|X_A|$-creature, and so $|X_A| < k$. We define a similar set $X_B \subseteq B$, and we observe that $C = N(X_A) \cap N(X_B)$.

Now, the following algorithm enumerates all minimal separators of $G$: for every pair of sets $X_A, X_B$ with $|X_A|,|X_B| < k$, compute $C = N(X_A) \cap N(X_B)$ and check whether $C$ is a minimal separator. Since $\binom{n}{i} \leq n^i$, we have $\binom{n}{0} + \dots + \binom{n}{k-1} \leq k n^{k-1}$. Therefore, the algorithm enumerates at most $\mathcal{O}(n^{2k-2})$ minimal separators in time $\mathcal{O}(n^{2k})$.
\end{proof}

We note that there exist graphs in $\C$ of arbitrarily large cliquewidth. In \cite{ALMRTV}, examples of even-hole-free graphs of arbitrarily large cliquewidth are presented. Those graphs are also diamond-free and they have no clique separators. (A \emph{diamond} is the graph with vertex set $\{a, b, c, d\}$ with all possible edges except $ab$.) However, they are not in $\C$ because they contain pyramids. In \cite{CTTV}, a procedure to obtain graphs in $\C$ with unbounded cliquewidth by modifying graphs defined in \cite{ALMRTV} is explained in detail. Moreover, those graphs contain arbitrarily large immature $k$-creatures, and so the main result of the current paper is not a corollary of Theorem \ref{thm:immature_k_creature}.

The rest of the paper is devoted to the proof of Theorem \ref{thm:main_thm}. In Section \ref{sec:decomposing_C}, we prove a useful theorem about star cutsets of graphs in $\C$. In Section \ref{sec:structure_proper}, we describe the structure of proper separators of graphs in $\C$. In Section \ref{sec:constructing_separators}, we construct a list of all minimal separators of graphs in $\C$ and prove Theorem \ref{thm:main_thm}.

\subsection*{Definitions}

Let $G = (V, E)$ be a graph. For $X \subseteq V(G)$, $G[X]$ denotes the induced subgraph of $G$ with vertex set $X$ and $G \setminus X$ denotes the induced subgraph of $G$ with vertex set $V(G) \setminus X$. We use induced subgraphs and their vertex sets interchangeably throughout the paper. We say that $G$ \emph{contains} a graph $H$ if $G$ has an induced subgraph isomorphic to $H$. A graph G is \emph{$H$-free} if it does not contain $H$. When $\mathcal{H}$ is a set of graphs, we say that $G$ is $\mathcal{H}$-free if $G$ is $H$-free for every $H \in \mathcal{H}$. For a graph $H$, we say that a set $X \subseteq V(G)$ is an $H$ in $G$ if $G[X]$ is isomorphic to $H$.

Let $X \subseteq V(G)$. The \emph{neighborhood} of $X$ in $G$, denoted by $N(X)$, is the set of all vertices in $V(G) \setminus X$ with a neighbor in $X$. The \emph{closed neighborhood} of $X$ in $G$, denoted $N[X]$, is given by $N[X] = N(X) \cup X$. For $u \in V(G)$, $N(u) = N(\{u\})$ and $N[u] = N[\{u\}]$. For $u \in V(G) \setminus X$, $N_X(u) = N(u) \cap X$. Let $Y \subseteq V(G)$ be disjoint from $X$. We say $X$ is \emph{complete} to $Y$ if every vertex in $X$ is adjacent to every vertex in $Y$, and $X$ is \emph{anticomplete} to $Y$ if every vertex in $X$ is non-adjacent to every vertex in $Y$.  Note that the empty set is complete and anticomplete to every $X \subseteq V(G)$. We say that a vertex $v$ is \emph{complete} (\emph{anticomplete}) to $X \subseteq V(G)$ if $\{v\}$ is complete (anticomplete) to $X$, and an edge $e = uv$ is \emph{complete} (\emph{anticomplete}) to $X$ if $\{u, v\}$ is complete (anticomplete) to $X$. 

A \emph{clique} in $G$ is a set of pairwise adjacent vertices, and an \emph{independent set} is a set of pairwise non-adjacent vertices. A \emph{triangle} is a clique of size three. A {\em path} in $G$ is an induced subgraph isomorphic to a graph $P$ with $k+1$ vertices $p_0, p_1, \dots, p_k$ and with $E(P) = \{p_ip_{i+1} : i \in \{0,\dots, k-1\}\}$. We write $P = p_0 \dd p_1 \dd \dots \dd p_k$ to denote a path with vertices $p_0, p_1, \dots, p_k$ in order. We say that $P$ is a path from $p_0$ to $p_k$. For a set $Y \subseteq V(G)$, if $P \setminus \{p_0,p_k\} \subseteq Y$, we say that $P$ is a path from $p_0$ to $p_k$ through $Y$. The \emph{length} of a path $P$ is the number of edges in $P$. A path is {\em odd} if its length is odd, and {\em even} otherwise. If $a, b \in P$, we denote by $aPb$ the subpath of $P$ from $a$ to $b$. For a path $P$ with ends $a, b$, the \emph{interior} of $P$, denoted $P^*$, is the set $V(P) \setminus \{a, b\}$. For an integer $k \geq 4$, a {\em hole of length $k$} in $G$ is an induced subgraph isomorphic to the $k$-vertex cycle $C_k$. A hole is {\em odd} if its length is odd, and {\em even} if its length is even.

If $X, Y, Z \subseteq V(G)$ are such that $X\cap Z= \emptyset$, we say that the path $P = p_0 \dd \hdots \dd p_k$ is a path from $X$ to $Z$ through $Y$ if $p_0 \in X$, $V(P) \setminus \{p_0\} \subseteq Y$, $V(P) \setminus \{p_k\}$ is anticomplete to $Z$, and $p_k$ has a neighbor in $Z$. When $X = \{p_0\}$, we say that $P$ is a path from $p_0$ to $Z$ through $Y$. Note that $P$ is disjoint from $Z$, and possibly $P=p_0$ (when $p_0$ is a vertex from $X$ with neighbors in $Z$). A path from $X$ to $Z$ through $Y$, when it exists, can be computed in time $\mathcal{O}(|V(G)|^2)$ as follows. In $G[X \cup Y \cup Z]$, compute a shortest path $q_0 \dd \hdots \dd q_k$ from $X$ to $Z$ by a breadth-first search (that also terminates when no such path exists). Then, the path $q_0 \dd \hdots \dd q_{k-1}$ is from $X$ to $Z$ through $Y$. Observe that this algorithm detects when no path from $X$ to $Z$ through $Y$ exists (when no connected component of $G[X \cup Y \cup Z]$ contains vertices of both $X$ and $Z$).

\section{Star cutsets in (theta, pyramid, prism, turtle)-free graphs}\label{sec:decomposing_C}

Let $G$ be a graph and $H$ a hole in $G$. A \emph{minor vertex} for $H$ is a vertex $u \in G \setminus H$ such that $u$ has neighbors in $H$ and $N_H(u)$ is contained in a three-vertex path. A vertex $v \in G \setminus H$ is a \emph{major vertex} for $H$ if $v$ has neighbors in $H$ and $v$ is not minor. A set $X \subseteq V(G)$ is a \emph{star cutset} of $G$ if $G \setminus X$ is not connected, and there exists $x \in X$ such that $x$ is complete to $X \setminus x$. We call $x$ the \emph{center} of the star cutset $X$. The goal of this section is to prove that major vertices for holes of graphs $G \in \C$ are the centers of star cutsets of $G$. The following two lemmas describe major and minor vertices.

\begin{lemma}
Let $G \in \C$. Every major vertex $u$ for a hole $H$ of $G$ has at least four neighbors in $H$ or has exactly three neighbors in $H$ that are pairwise non-adjacent.
\label{lemma:1}
\end{lemma}

\begin{proof}
Let $u$ be a major vertex for a hole $H$. Suppose $u$ has exactly two neighbors in $H$. The neighbors of $u$ are non-adjacent, because they are not contained in a three-vertex path. Thus, $H \cup \{u\}$ forms a theta between the neighbors of $u$ in $H$, a contradiction. If $u$ has exactly three neighbors in $H$, they are pairwise non-adjacent, otherwise $H \cup \{u\}$ forms a pyramid. 
\end{proof}

\begin{lemma}
Let $G \in \C$. Then, every minor vertex $u$ for a hole $H$ of $G$ satisfies one of the following: 
\begin{itemize}
\itemsep-0.2em
\item $u$ has a unique neighbor in $H$ (we say that $u$ is a \emph{pendant} of $H$)
\item $u$ has two adjacent neighbors in $H$ (we say that $u$ is a \emph{cap} of $H$)
\item $u$ has three neighbors in $H$ which induce a path $xyz$ (we say that $u$ is a \emph{clone} of $y$ in $H$)
\end{itemize}
\vspace{-\baselineskip}
\vspace{0.2cm}
\label{lemma:2}
\end{lemma}

\begin{proof}
Since $u$ is minor, $N_H(u)$ is contained in a three-vertex path. Therefore, the only possibility not listed above is that $u$ has two non-adjacent neighbors in $H$, in which case $H \cup \{u\}$ is a theta.
\end{proof}

Suppose $H$ is a hole of a graph $G$ and $u$ is a clone of $y$ in $H$. We denote by $H_{u \setminus y}$ the hole of $G$ induced by $(V(H) \setminus \{y\}) \cup \{u\}$. Note that $y$ is a clone of $u$ in $H_{u \setminus y}$.

\begin{lemma}
Let $H$ be a hole in a graph $G \in \C$, let $v$ be a major vertex for $H$, let $u$ be a clone of $y$ in $H$, and suppose $uv \in E(G)$. If $yv \not \in E(G)$, then $u$ and $v$ have a common neighbor in $H$.
\label{lemma:3}
\end{lemma}

\begin{proof}
Let $N_H(u) = \{x, y, z\}$. Suppose $uv \in E(G)$ and $yv \not \in E(G)$. Because $v$ is major, by Lemma \ref{lemma:1}, $v$ has at least three neighbors in $H$. If $v$ is anticomplete to $\{x, z\}$, then $H \cup \{u, v\}$ is a $uv$-turtle in $G$, a contradiction. Therefore, $v$ is adjacent to at least one of $x, z$, and so $u$ and $v$ have a common neighbor in $H$.
\end{proof}

Let $u \in G \setminus H$ be a vertex with at least two neighbors in $H$. A \emph{$u$-sector} is a path $P = u' \hdots u''$ such that $P \subseteq H$, $u'$ and $u''$ are neighbors of $u$, and $P^*$ is anticomplete to $u$.

\begin{lemma}
Let $u$ and $v$ be two non-adjacent major vertices for a hole $H$ of a graph $G \in \C$. Let $P = u' \hdots u''$ be a $u$-sector of $H$. Then one of the following holds:
\begin{enumerate}[(i)]
\itemsep-0.2em
\item $P$ contains at most one neighbor of $v$, and if it has one, it is either $u'$ or $u''$,
\item $u'u'' \in E(G)$ and $v$ is adjacent to both $u'$ and $u''$,
\item $P$ contains at least 3 neighbors of $v$,
\item $H \cup \{u, v\}$ is a cube (see Figure \ref{fig:cube}).
\end{enumerate}
\vspace{-\baselineskip}
\vspace{0.2cm}
\label{lemma:4}
\end{lemma}

\begin{figure}[ht]
\begin{center}
\begin{tikzpicture}[scale=0.3]

\node[inner sep=2.5pt, fill=black, circle] at (0, 6)(v1){}; 
\node[inner sep=2.5pt, fill=black, circle] at (4, 6)(v2){}; 
\node[inner sep=2.5pt, fill=black, circle] at (6, 4)(v3){}; 
\node[inner sep=2.5pt, fill=black, circle] at (6, 0)(v4){}; 
\node[inner sep=2.5pt, fill=black, circle] at (2, 0)(v5){}; 
\node[inner sep=2.5pt, fill=black, circle] at (0, 2)(v6){}; 
\node[inner sep=2.5pt, fill=black, circle] at (2, 4)(v7){};
\node[inner sep=2.5pt, fill=black, circle] at (4, 2)(v8){};

\node[inner sep=2.5pt, fill=white, circle] at (1, -0.2)(v21){}; 

\draw[black, thick] (v1) -- (v2);
\draw[black, thick] (v2) -- (v3);
\draw[black, thick] (v3) -- (v4);
\draw[black, thick] (v4) -- (v5);
\draw[black, thick] (v5) -- (v6);
\draw[black, thick] (v1) -- (v6);
\draw[black, thick] (v1) -- (v7);
\draw[black, thick] (v3) -- (v7);
\draw[black, thick] (v5) -- (v7);
\draw[black, thick] (v2) -- (v8);
\draw[black, thick] (v4) -- (v8);
\draw[black, thick] (v6) -- (v8);

\end{tikzpicture}
\end{center}
\vspace{-0.4cm}
\caption{The cube graph}
\label{fig:cube}
\end{figure}
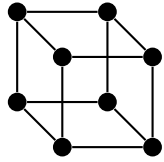

\begin{proof}
Let $R = H \setminus P$, let $x$ be the neighbor of $u'$ in $R$, and let $y$ be the neighbor of $u''$ in $R$. First, suppose that $P$ contains exactly two neighbors of $v$. We may assume that $u'u'' \not \in E(G)$, otherwise outcome (ii) holds. Let the neighbors of $v$ in $P$ be given by $v', v''$. Then, $G$ contains a theta between $v'$ and $v''$ in $P \cup \{u, v\}$ unless $v'v'' \in E(G)$, so $v'v'' \in E(G)$. Since either $v' \neq u', u''$ or $v'' \neq u', u''$, we may assume that $v' \neq u', u''$ and $v'$ is between $u'$ and $v''$ in $P$ (possibly $v'' = u''$). If both $u$ and $v$ have neighbors in $R^*$, then $G$ contains a pyramid from $u$ to $vv'v''$ via paths through $u'$, $u''$, and $R^*$, so at least one of $u, v$ has no neighbors in $R^*$. By Lemma \ref{lemma:1}, either $N_H(u) = \{x, u', u'', y\}$ or $N_H(v) = \{x, v', v'', y\}$. Since $v$ is major, $v$ has a neighbor in $R \setminus y$. Let $v'''$ be the neighbor of $v$ in $R \setminus y$ closest to $x$. If $ux \in E(G)$, then $G$ contains a prism from $uxu'$ to $v''vv'$ through the paths $u \dd u'' \dd P \dd v''$, $u' \dd P \dd v'$, and $x \dd R \dd v''' \dd v$, so $ux \notin E(G)$ and thus $N_H(v) = \{x, v', v'', y\}$. Now, $G$ contains a pyramid from $u'$ to $vv'v''$ via the paths $u' \dd x \dd v$, $u' \dd u \dd u'' \dd P \dd v''$, and $u' \dd P \dd v'$, a contradiction.

Now, suppose $P$ contains exactly one neighbor $v'$ of $v$. We may assume that $v' \neq u', u''$, otherwise outcome (i) holds. If $u$ and $v$ both have neighbors in $R^*$, then $G$ contains a theta between $u$ and $v'$ through $u'$, $u''$, and $R^* \cup \{v\}$, so at least one of $u, v$ has no neighbors in $R^*$. By Lemma \ref{lemma:1}, either $N_H(u) = \{x, u', u'', y\}$ or $N_H(v) = \{x, v', y\}$. Since $v$ is major, $v$ has a neighbor in $R \setminus y$. Let $v''$ be the neighbor of $v$ in $R$ closest to $x$. If $ux \in E(G)$, then $G$ contains a pyramid from $v'$ to $u'ux$ through the paths $v' \dd P \dd u'$, $v' \dd v \dd v'' \dd R \dd x$, and $v' \dd P \dd u'' \dd u$, so $ux \not \in E(G)$. By symmetry, $uy \not \in E(G)$. It follows that $N_H(v) = \{x, v', y\}$. Because $u$ is major, $u$ has a neighbor in $R^*$. Let $u'''$ be the neighbor of $u$ in $R^*$ closest to $y$. Then $G$ contains a theta between $u$ and $y$ through the paths $u \dd u'' \dd y$, $u \dd u' \dd x \dd v \dd y$, and $u \dd u''' \dd R \dd y$, unless $xu''' \in E(G)$. Similarly, $G$ contains a theta between $u$ and $x$ unless $yu''' \in E(G)$. If $u'v' \notin E(G)$, then $G$ contains a theta between $u'$ and $v'$ through the paths $u' \dd P \dd v'$, $u' \dd x \dd v \dd v'$, and $u' \dd u \dd u'' \dd P \dd v'$, so $v'u' \in E(G)$. Similarly, $v'u'' \in E(G)$. Now, $H \cup \{u, v\}$ is a cube, and so outcome (iv) holds. This completes the proof.
\end{proof}

Let $H$ be a hole in a graph $G$ and let $u, v \in V(G) \setminus V(H)$. We say that $u$ and $v$ are \emph{nested} with respect to $H$ if there exist distinct $a, b \in V(H)$ such that one $ab$-path of $H$ contains all the neighbors of $u$ and the other $ab$-path of $H$ contains all the neighbors of $v$. Note that pendants, caps, and vertices of $G \setminus H$ with no neighbor in $H$ are nested with all other vertices of $G\setminus H$. The vertices $u$ and $v$ are \emph{strictly nested} with respect to $H$ if $u$ and $v$ are nested with respect to $H$ and $N_H(u) \cap N_H(v) = \emptyset$.

\begin{lemma}
Let $H$ be a hole in a graph $G \in \C$, let $u$ and $v$ be major or clones for $H$ such that $u$ and $v$ are nested, and suppose $uv \in E(G)$. Then, $u$ and $v$ have a common neighbor in $H$.
\label{lemma:5}
\end{lemma}

\begin{proof} 
Since $u$ and $v$ are major or clones, $u$ and $v$ have at least three neighbors in $H$. If $u$ and $v$ have no common neighbors in $H$, then $H \cup \{u, v\}$ is a $uv$-turtle in $G$, a contradiction.
\end{proof}

Two vertices $u$ and $v$ not in $H$ that are not nested with respect to $H$ are said to \emph{cross}. The following lemmas characterize the behavior of vertices that cross. 

\begin{lemma}
Let $H$ be a hole in a graph $G \in \C$ and let $u$ and $v$ be two vertices not in $H$ that cross. Then, one of the following holds: 
\begin{enumerate}[(i)]
\itemsep-0.2em
    \item $H$ contains four distinct vertices $u', v', u'', v''$ that appear in this order along $H$ such that $u', u'' \in N_H(u)$ and $v', v'' \in N_H(v)$,
    \item $N_H(u) = N_H(v)$, $N_H(u)$ is an independent set, $|N_H(u)| = 3$, and $uv \in E(G)$,
    \item $N_H(u) = N_H(v)$, both $u$ and $v$ are clones in $H$, and $uv \in E(G)$.
\end{enumerate}
\vspace{-\baselineskip}
\vspace{0.2cm}
\label{lemma:6}
\end{lemma}

\begin{proof}
Since $u$ and $v$ are not nested, $u$ and $v$ are major or clones for $H$. Suppose $u$ and $v$ are both clones. Since $u$ and $v$ cross, it follows that either $u$ and $v$ are clones of adjacent vertices, so outcome (i) holds, or $u$ and $v$ are clones of the same vertex. Let $u$ and $v$ both be clones of $y$ and let $N_H(u) = N_H(v) = \{x, y, z\}$. Then, $G$ contains a theta between $x$ and $z$ in $H_{u \setminus y} \cup \{v\}$ unless $uv \in E(G)$, so outcome (iii) holds. Now, suppose $u$ is a clone of $y$ and $v$ is major. Because $u$ and $v$ cross, $vy \in E(G)$, and because $v$ is major, $v$ has at least one neighbor in $V(H) \setminus N_H(u)$. Therefore, outcome (i) holds. Finally, suppose $u$ and $v$ are both major. Assume that $u$ has a neighbor $x$ in $H$ such that $vx \not \in E(G)$. Then, $x$ is contained in a $v$-sector $P = v' \hdots v''$ of $H$. Because $u$ and $v$ cross, $u$ has a neighbor in $H \setminus P$, so outcome (i) holds. Therefore, we may assume that $N_H(u) = N_H(v)$. If $|N_H(u)| > 3$, outcome (i) holds, so $|N_H(u)| = 3$. Because $u$ is major and $|N_H(u)| = 3$, it follows from Lemma \ref{lemma:1} that $N_H(u)$ is an independent set. Let $N_H(u) = \{x, y, z\}$. Then, $G$ contains a theta between $x$ and $y$ in $(H \setminus \{z\}) \cup \{u, v\}$ unless $uv \in E(G)$, so outcome (ii) holds.
\end{proof}

Let $H = h_1 \dd h_2 \dd \dots \dd h_k \dd h_1$ be a hole in a graph $G \in \C$ and let $u, v \in V(G) \setminus V(H)$ be two non-adjacent major vertices for $H$. The following lemma shows that if $u$ and $v$ cross, then $H \cup \{u, v\}$ is a \emph{major non-adjacent cross (MNC) configuration}. We describe MNC configurations as follows (see also Figure \ref{fig:MNC_config}).
\begin{itemize}
\itemsep-0.1em
\item MNC configuration (1): $k=4$, and $\{u ,v\}$ are complete to $H$.

\item MNC configuration (2): $k=5$, and $\{u ,v\}$ are complete to $H$.

\item MNC configuration (3): $k=6$, $N_H(u) = \{h_1, h_3, h_5\}$, and $N_H(v) = \{h_2, h_4, h_6\}$.

\item MNC configuration (4): $\{u, v\}$ is complete to $\{h_1, h_2, h_3, h_4\}$, $v$ has no other neighbors in $H$, and $u$ has at least one other neighbor in $H$.

\medskip

Let $3 < i < k-1$. Let $H_1$ be the path from $h_1$ to $h_{i+1}$ in $H \setminus \{h_2\}$ and let $H_2$ be the path from $h_2$ to $h_i$ in $H \setminus \{h_1\}$.

\medskip

\item MNC configuration (5): $\{u, v\}$ is complete to $\{h_1, h_2, h_i, h_{i+1}\}$, $u$ and $v$ both have other neighbors in $H$, $N_H(u) \subseteq H_1 \cup \{h_2, h_i\}$, and $N_H(v) \subseteq H_2 \cup \{h_1, h_{i+1}\}$.

\item MNC configuration (6): $\{u, v\}$ is complete to $\{h_1, h_2, h_i, h_{i+1}\}$, $v$ has no other neighbors in $H$, and $u$ has neighbors both in $H_1^*$ and $H_2^*$.

\item MNC configuration (7): $\{u, v\}$ is complete to $\{h_1, h_2, h_i\}$, $u$ and $v$ both have other neighbors in $H$, $N_H(u) \subseteq H_1 \cup \{h_2, h_i\}$, and $N_H(v) \subseteq H_2 \cup \{h_1\}$.

\item MNC configuration (8): $\{u, v\}$ is complete to $\{h_1, h_2\}$, $u$ and $v$ both have other neighbors in $H$, $N_H(u) \subseteq H_1 \cup \{h_2\}$, and $N_H(v) \subseteq H_2 \cup \{h_1\}$.
\end{itemize}

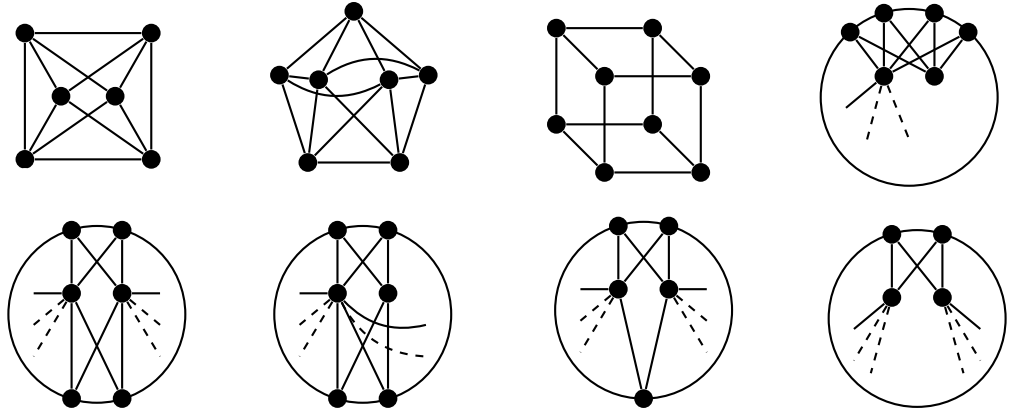
\begin{figure}[ht]
\begin{center}
\begin{tikzpicture}[scale=0.24]

\node[inner sep=2.5pt, fill=black, circle] at (0,0)(v1){}; 
\node[inner sep=2.5pt, fill=black, circle] at (7, 0)(v2){}; 
\node[inner sep=2.5pt, fill=black, circle] at (7, -7)(v3){}; 
\node[inner sep=2.5pt, fill=black, circle] at (0, -7)(v4){};
\node[inner sep=2.5pt, fill=black, circle] at (2, -3.5)(v6){}; 
\node[inner sep=2.5pt, fill=black, circle] at (5, -3.5)(v7){}; 

\node[inner sep=2.5pt, fill=white, circle] at (0, -8)(v21){}; 

\draw[black, thick] (v1) -- (v2);
\draw[black, thick] (v2) -- (v3);
\draw[black, thick] (v3) -- (v4);
\draw[black, thick] (v4) -- (v1);
\draw[black, thick] (v1) -- (v6);
\draw[black, thick] (v2) -- (v6);
\draw[black, thick] (v3) -- (v6);
\draw[black, thick] (v4) -- (v6);
\draw[black, thick] (v1) -- (v7);
\draw[black, thick] (v2) -- (v7);
\draw[black, thick] (v3) -- (v7);
\draw[black, thick] (v4) -- (v7);

\end{tikzpicture}
\hspace{1.2cm}
\begin{tikzpicture}[scale=0.26]

\node[inner sep=2.5pt, fill=black, circle] at (0,4.5)(v1){}; 
\node[inner sep=2.5pt, fill=black, circle] at (3.804226, 1.2361)(v2){}; 
\node[inner sep=2.5pt, fill=black, circle] at (2.351141, -3.2361)(v3){}; 
\node[inner sep=2.5pt, fill=black, circle] at (-2.351141, -3.2361)(v4){}; 
\node[inner sep=2.5pt, fill=black, circle] at (-3.804226, 1.2361)(v5){};
\node[inner sep=2.5pt, fill=black, circle] at (1.8, 1)(v6){}; 
\node[inner sep=2.5pt, fill=black, circle] at (-1.8, 1)(v7){}; 

\node[inner sep=2.5pt, fill=white, circle] at (0, -4)(v21){}; 

\draw[black, thick] (v1) -- (v2);
\draw[black, thick] (v1) -- (v5);
\draw[black, thick] (v2) -- (v3);
\draw[black, thick] (v3) -- (v4);
\draw[black, thick] (v4) -- (v5);
\draw[black, thick] (v1) -- (v6);
\draw[black, thick] (v2) -- (v6);
\draw[black, thick] (v3) -- (v6);
\draw[black, thick] (v4) -- (v6);
\draw[black, thick] (v1) -- (v7);
\draw[black, thick] (v3) -- (v7);
\draw[black, thick] (v4) -- (v7);
\draw[black, thick] (v5) -- (v7);

\draw[black, thick](v2) edge [bend right=30] (v7);
\draw[black, thick](v5) edge [bend right=30] (v6);

\end{tikzpicture}
\hspace{1.2cm}
\begin{tikzpicture}[scale=0.32]

\node[inner sep=2.5pt, fill=black, circle] at (0, 6)(v1){}; 
\node[inner sep=2.5pt, fill=black, circle] at (4, 6)(v2){}; 
\node[inner sep=2.5pt, fill=black, circle] at (6, 4)(v3){}; 
\node[inner sep=2.5pt, fill=black, circle] at (6, 0)(v4){}; 
\node[inner sep=2.5pt, fill=black, circle] at (2, 0)(v5){}; 
\node[inner sep=2.5pt, fill=black, circle] at (0, 2)(v6){}; 
\node[inner sep=2.5pt, fill=black, circle] at (2, 4)(v7){};
\node[inner sep=2.5pt, fill=black, circle] at (4, 2)(v8){};

\node[inner sep=2.5pt, fill=white, circle] at (1, -0.2)(v21){}; 

\draw[black, thick] (v1) -- (v2);
\draw[black, thick] (v2) -- (v3);
\draw[black, thick] (v3) -- (v4);
\draw[black, thick] (v4) -- (v5);
\draw[black, thick] (v5) -- (v6);
\draw[black, thick] (v1) -- (v6);
\draw[black, thick] (v1) -- (v7);
\draw[black, thick] (v3) -- (v7);
\draw[black, thick] (v5) -- (v7);
\draw[black, thick] (v2) -- (v8);
\draw[black, thick] (v4) -- (v8);
\draw[black, thick] (v6) -- (v8);

\end{tikzpicture}
\hspace{1.2cm}
\begin{tikzpicture}[scale=0.28]

\node[inner sep=2.5pt, fill=black, circle] at (-1.2, 4)(v2){}; 
\node[inner sep=2.5pt, fill=black, circle] at (1.2, 4)(v3){}; 

\node[inner sep=2.5pt, fill=black, circle] at (-2.8, 3.1)(v1){}; 
\node[inner sep=2.5pt, fill=black, circle] at (2.8, 3.1)(v4){}; 

\node[inner sep=2.5pt, fill=black, circle] at (-1.2, 1)(v5){}; 
\node[inner sep=2.5pt, fill=black, circle] at (1.2, 1)(v6){};

\draw[black, thick] (v5) -- (v1);
\draw[black, thick] (v6) -- (v1);
\draw[black, thick] (v5) -- (v4);
\draw[black, thick] (v6) -- (v4);
\draw[black, thick] (v5) -- (v2);
\draw[black, thick] (v6) -- (v2);
\draw[black, thick] (v5) -- (v3);
\draw[black, thick] (v6) -- (v3);

\draw[black, thick] (v5) -- (-3, -0.5);
\draw[black, dashed, thick] (v5) -- (-2, -2);
\draw[black, dashed, thick] (v5) -- (0, -2);

\draw[black, thick] (0,0) circle (4.2cm);

\end{tikzpicture}
\hspace{1cm}
\begin{tikzpicture}[scale=0.28]

\node[inner sep=2.5pt, fill=black, circle] at (-1.2, 4)(v2){}; 
\node[inner sep=2.5pt, fill=black, circle] at (1.2, 4)(v3){}; 

\node[inner sep=2.5pt, fill=black, circle] at (-1.2, -4)(v1){}; 
\node[inner sep=2.5pt, fill=black, circle] at (1.2, -4)(v4){}; 

\node[inner sep=2.5pt, fill=black, circle] at (-1.2, 1)(v5){}; 
\node[inner sep=2.5pt, fill=black, circle] at (1.2, 1)(v6){}; 

\draw[black, thick] (v5) -- (v1);
\draw[black, thick] (v6) -- (v1);
\draw[black, thick] (v5) -- (v4);
\draw[black, thick] (v6) -- (v4);
\draw[black, thick] (v5) -- (v2);
\draw[black, thick] (v6) -- (v2);
\draw[black, thick] (v5) -- (v3);
\draw[black, thick] (v6) -- (v3);

\draw[black, dashed, thick] (v5) -- (-3, -0.5);
\draw[black, dashed, thick] (v5) -- (-3, -2);
\draw[black, thick] (v5) -- (-3, 1);

\draw[black, dashed, thick] (v6) -- (3, -0.5);
\draw[black, dashed, thick] (v6) -- (3, -2);
\draw[black, thick] (v6) -- (3, 1);

\draw[black, thick] (0,0) circle (4.2cm);

\end{tikzpicture}
\hspace{0.9cm}
\begin{tikzpicture}[scale=0.28]

\node[inner sep=2.5pt, fill=black, circle] at (-1.2, 4)(v2){}; 
\node[inner sep=2.5pt, fill=black, circle] at (1.2, 4)(v3){}; 

\node[inner sep=2.5pt, fill=black, circle] at (-1.2, -4)(v1){}; 
\node[inner sep=2.5pt, fill=black, circle] at (1.2, -4)(v4){}; 

\node[inner sep=2.5pt, fill=black, circle] at (-1.2, 1)(v5){}; 
\node[inner sep=2.5pt, fill=black, circle] at (1.2, 1)(v6){}; 

\node[inner sep=2.5pt, fill=white, circle] at (0, 5.5)(v21){}; 

\draw[black, thick] (v5) -- (v1);
\draw[black, thick] (v6) -- (v1);
\draw[black, thick] (v5) -- (v4);
\draw[black, thick] (v6) -- (v4);
\draw[black, thick] (v5) -- (v2);
\draw[black, thick] (v6) -- (v2);
\draw[black, thick] (v5) -- (v3);
\draw[black, thick] (v6) -- (v3);

\draw[black, dashed, thick] (v5) -- (-3, -0.5);
\draw[black, dashed, thick] (v5) -- (-3, -2);
\draw[black, thick] (v5) -- (-3, 1);

\draw[black, thick] (v5) edge [bend right=30] (3, -0.5);
\draw[black, dashed, thick] (v5) edge [bend right=30] (3, -2);

\draw[black, thick] (0,0) circle (4.2cm);

\end{tikzpicture}
\hspace{1.1cm}
\begin{tikzpicture}[scale=0.28]

\node[inner sep=2.5pt, fill=black, circle] at (0, -4.2)(v1){}; 
\node[inner sep=2.5pt, fill=black, circle] at (-1.2, 4)(v2){}; 
\node[inner sep=2.5pt, fill=black, circle] at (1.2, 4)(v3){}; 

\node[inner sep=2.5pt, fill=black, circle] at (-1.2, 1)(v5){}; 
\node[inner sep=2.5pt, fill=black, circle] at (1.2, 1)(v6){}; 

\draw[black, thick] (v5) -- (v1);
\draw[black, thick] (v6) -- (v1);
\draw[black, thick] (v5) -- (v2);
\draw[black, thick] (v6) -- (v2);
\draw[black, thick] (v5) -- (v3);
\draw[black, thick] (v6) -- (v3);

\draw[black, dashed, thick] (v5) -- (-3, -0.5);
\draw[black, dashed, thick] (v5) -- (-3, -2);
\draw[black, thick] (v5) -- (-3, 1);

\draw[black, dashed, thick] (v6) -- (3, -0.5);
\draw[black, dashed, thick] (v6) -- (3, -2);
\draw[black, thick] (v6) -- (3, 1);

\draw[black, thick] (0,0) circle (4.2cm);

\end{tikzpicture}
\hspace{1cm}
\begin{tikzpicture}[scale=0.28]

\node[inner sep=2.5pt, fill=black, circle] at (-1.2, 4)(v2){}; 
\node[inner sep=2.5pt, fill=black, circle] at (1.2, 4)(v3){}; 

\node[inner sep=2.5pt, fill=black, circle] at (-1.2, 1)(v5){}; 
\node[inner sep=2.5pt, fill=black, circle] at (1.2, 1)(v6){}; 

\draw[black, thick] (v5) -- (v2);
\draw[black, thick] (v6) -- (v2);
\draw[black, thick] (v5) -- (v3);
\draw[black, thick] (v6) -- (v3);

\draw[black, thick] (v5) -- (-3, -0.5);
\draw[black, dashed, thick] (v5) -- (-3, -2);
\draw[black, dashed, thick] (v5) -- (-2.2, -2.6);

\draw[black, thick] (v6) -- (3, -0.5);
\draw[black, dashed, thick] (v6) -- (3, -2);
\draw[black, dashed, thick] (v6) -- (2.2, -2.6);

\draw[black, thick] (0,0) circle (4.2cm);

\end{tikzpicture}
\end{center}
\vspace{-0.1cm}
\caption{MNC configurations (dashed lines represent possible edges)}
\label{fig:MNC_config}
\end{figure}

\begin{lemma}
Let $H$ be a hole in a graph $G \in \C$ and suppose that $u$ and $v$ are two major vertices for $H$. If $uv \in E(G)$, then either $u$ and $v$ cross or $u$ and $v$ have a common neighbor in $H$. If $u$ and $v$ cross, then either $uv \in E(G)$ or $G$ contains an MNC configuration.
\label{lemma:7}
\end{lemma}

\begin{proof}
The first statement follows from Lemma \ref{lemma:5}. It remains to prove the second statement. Assume that $u$ and $v$ cross and $uv \not \in E(G)$. Let $H = h_1 \dd \hdots \dd h_k \dd h_1$.

\medskip

\noindent \emph{(1) No three common neighbors of $u$ and $v$ in $H$ form an independent set.}

If $u$ and $v$ have three common neighbors $x, y, z \in V(H)$ such that $\{x, y, z\}$ is an independent set, then $\{u, v, x, y, z\}$ is a theta in $G$ between $u$ and $v$, a contradiction. This proves (1).

\medskip

\noindent \emph{(2) Suppose there exist $i, j \in \{1, \hdots, k\}$ with $j > i$ such that $h_ih_j \not \in E(G)$, $\{u, v\}$ is complete to $\{h_i, h_j\}$, and there are no common neighbors of $u$ and $v$ in $P = h_{j+1}\dd h_{j+2}\dd \hdots \dd h_{i-1}$. Then, $u$ and $v$ do not both have neighbors in $P$.}

If both $u$ and $v$ have neighbors in $P^*$, then $G$ contains a theta between $u$ and $v$, through $h_i$, $h_j$, and $P^*$. Therefore, we may assume that $h_{j+1}$ is adjacent to $u$. If $v$ has a neighbor in $P^*$, let $v'$ be the neighbor of $v$ in $P^*$ closest to $h_{j+2}$. Then, $G$ contains a pyramid from $v$ to $uh_jh_{j+1}$ through the paths $v \dd h_j$, $v \dd h_i \dd u$, and $v \dd v' \dd P \dd h_{j+1}$, a contradiction. Therefore, we may assume that $N_P(v) = \{h_{i-1}\}$. Then, $G$ contains a prism from $uh_jh_{j+1}$ to $h_ivh_{i-1}$ through the paths $u \dd h_i$, $h_j \dd v$, and $h_{j+1} \dd P \dd h_{i-1}$, a contradiction. This proves (2).

\medskip

Let $N = N_H(u) \cap N_H(v)$.

\medskip

\noindent \emph{(3) We may assume that $R = H[N]$ is a subpath of $H$ of length at most one.}

If $N = V(H)$, then by $(1)$, $H \cup \{u, v\}$ is MNC configuration (1) or (2). So, we may assume that $N \neq V(H)$. Suppose first that $R$ is a subpath of $H$ of length greater than one. By $(1)$, $R$ is of length at most three. If it is of length two, then by $(2)$, one of $u, v$ is a clone, a contradiction. Suppose $R$ is of length three, say $R=h_i \dd h_{i+1}\dd h_{i+2} \dd h_{i+3}$. Then, by $(2)$, at most one of $u,v$ has a neighbor in $H \setminus R$. If one of $u,v$ has a neighbor in $H \setminus R$, then $H \cup \{u, v\}$ is MNC configuration~(4), otherwise $G$ contains a theta between $h_i$ and $h_{i+3}$ through $u$, $v$, and $H \setminus R$.

Next, suppose that $R$ is not a subpath of $H$. By (1), it follows that $R$ is the disjoint union of two subpaths $Q_1, Q_2$ of $H$ of length at most one. Let $Q_1 = h_1 \dd \hdots \dd h_i$ and $Q_2 = h_s \dd \hdots \dd h_t$, where $k > t \geq s \geq i+2$, $i \leq 2$, and $t \leq s+1$. By (2), we may assume that $v$ has no neighbors in $h_{t+1} \dd h_{t+2} \dd \hdots \dd h_k$, and that not both $u$ and $v$ have neighbors in $h_{i+1} \dd \hdots \dd h_{s-1}$. Since $u$ and $v$ cross, at least one of $Q_1$ and $Q_2$ has length one, so we may assume $i = 2$. 

Note that at least one of $u$ and $v$ has a neighbor in $h_{t+1} \dd \dots \dd h_k$, otherwise $G$ contains a theta between $h_1$ and $h_t$ through $H \setminus \{h_2, h_s\}$, $u$, and $v$. Similarly, at least one of $u$ and $v$ has a neighbor in $h_{i+1} \dd \dots \dd h_{s-1}$. It follows that if $t=s+1$, then $G$ contains MNC configuration (5) or (6), and if $t = s$, then $G$ contains MNC configuration (7). This proves (3).

\medskip

\noindent \emph{(4) We may assume that $N = \emptyset$.}

We may assume that $H \cup \{u, v\}$ is not a cube (MNC configuration (3)) or MNC configuration~(8). By (3), we may assume that $R = h_1$ or $R = h_1h_2$. Let $i = 1$ if $R = h_1$, and $i = 2$ if $R = h_1h_2$. We claim that we may assume that $H \setminus R$ contains three distinct vertices $u', v', u''$ such that $h_i, u', v', u'', h_1$ appear in this order and $u', u'' \in N_H(u)$, $v' \in N_H(v)$. If $i=2$, then this follows from the assumption that $G$ does not contain MNC configuration~(8), and if $i=1$, it follows from Lemma \ref{lemma:6}. Let $u', v', u''$ be chosen such that the path from $u'$ to $u''$ in $H \setminus R$ is a $u$-sector. Let $H_a$ be the path from $h_i$ to $v'$ in $H \setminus \{u''\}$, let $H_b$ be the path from $h_1$ to $v'$ in $H \setminus \{u'\}$, and let $H_u$ be the path from $u'$ to $u''$ in $H \setminus R$. Since $H_a$ contains a $v$-sector that contains $u'$, it follows from Lemma \ref{lemma:4} that there are at least three neighbors of $u$ in $H_a$, and so there at least two neighbors of $u$ in $H_a^*$. Since $H_u$ is a $u$-sector, there is a neighbor of $u$ in $H_a^*$ between $h_i$ and $u'$. Similarly, there are at least three neighbors of $u$ in $H_b$, and so at least two neighbors of $u$ in $H_b^*$. Since $H_u$ is a $u$-sector, there is a neighbor of $u$ in $H_b^*$ between $h_1$ and $u''$. Finally, since $H_u$ is a $u$-sector that contains $v'$, there are at least three neighbors of $v$ in $H_u^*$. Let $v_1$ be the neighbor of $v$ in $H_u^*$ closest to $u'$, and let $v_2$ be the neighbor of $v$ in $H_u^*$ closest to $u''$. Then, $G$ contains a theta from $u$ to $v$, through $u \dd h_1 \dd v$, $u \dd u' \dd H_u \dd v_1 \dd v$, and $u \dd u'' \dd H_u \dd v_2 \dd v$. This proves (4).

\medskip

It follows from Lemma \ref{lemma:6} and (4) that $H$ contains four distinct vertices $u', v', u'', v''$ that appear in that order along $H$, such that $u', u'' \in N_H(u)$ and $v', v'' \in N_H(v)$. By (4), $u$ and $v$ have no common neighbors in $H$, and so every neighbor of $u$ in $H$ is in the interior of a $v$-sector and every neighbor of $v$ in $H$ is in the interior of a $u$-sector. Let $u', v', u'', v''$ be chosen so that $P = u' \hdots u''$ is a $u$-sector and $Q = v' \hdots v''$ is a $v$-sector. We may assume that $H \cup \{u,v\}$ is not a cube (i.e. MNC configuration (3)). Because $v'$ is in $P^*$, it follows from Lemma \ref{lemma:4} that there are at least three neighbors of $v$ in $P^*$. Since there are no neighbors of $v$ in $Q^*$, we may assume that $v_1, v_2$ are neighbors of $v$ in $P^*$ between $u'$ and $v'$ in that order. Similarly, there are at least three neighbors of $u$ in $Q^*$. Since there are no neighbors of $u$ in $P^*$, we may assume that $u_1, u_2$ are neighbors of $u$ in $Q^*$ between $u''$ and $v''$ in that order. Finally, because $v''$ is in the interior of a $u$-sector and $Q = v' \hdots v''$ is a $v$-sector, there is another neighbor of $v$ between $v''$ and $u'$. Then, $G$ contains a theta between $u$ and $v$ through the paths $u \dd u_2 \dd Q \dd v'' \dd v$, $u \dd u'' \dd Q \dd v' \dd v$, and $u \dd u' \dd P \dd v_1 \dd v$, a contradiction. 
\end{proof}

\begin{lemma}
Let $H$ be a hole of length greater than six in a graph $G \in \C$ and suppose $u$ and $v$ are non-adjacent vertices of $G \setminus H$ that cross. Then, $H$ contains a $\{u,v\}$-complete edge.
\label{lemma:8}
\end{lemma}

\begin{proof}
Since $u$ and $v$ cross, $u$ and $v$ are major or clones. If $u$ and $v$ are both major, it follows from Lemma \ref{lemma:7} that $H \cup \{u, v\}$ is MNC configuration (4), (5), (6), (7), or (8), so $H$ contains a $\{u, v\}$-complete edge. Now, suppose $u$ is a clone of $y$ in $H$ and $N_H(u) = \{x, y, z\}$. Because $u$ and $v$ cross, it follows that $vy \in E(G)$. We may assume that $xv, zv \notin E(G)$ since otherwise $H$ contains a $\{u,v\}$-complete edge. Note that $x \dd u \dd z$ is a subpath of $H_{u \setminus y}$ that contains all the neighbors of $y$ in $H_{u \setminus y}$ and no neighbors of $v$. If $v$ has at least three neighbors in $H_{u\setminus y}$, then $G$ contains a $vy$-turtle. So $v$ has two neighbors in $H_{u\setminus y}$, say $v_1$ and $v_2$, and hence three neighbors in $H$. By Lemma \ref{lemma:1} applied to $H$ and $v$, $v_1v_2 \notin E(G)$. But then $H_{u\setminus y}$ and $v$ form a theta, a contradiction.
\end{proof}

The following lemma describes the behavior of paths whose endpoints are nested with respect to $H$ and whose internal vertices are anticomplete to $H$.

\begin{lemma}
Let $H$ be a hole in a graph $G \in \C$ and let $P = u \dd \hdots \dd v$ be a path of length at least 1, vertex-disjoint from $H$, such that $u$ and $v$ have neighbors in $H$ and are nested with respect to $H$, and no internal vertex of $P$ has a neighbor in $H$. Then, $u$ and $v$ have a common neighbor in $H$, or $u$ and $v$ are both pendants of $H$ with adjacent neighbors in $H$. In particular, if $u$ and $v$ are strictly nested with respect to $H$, then $u$ and $v$ are both pendants of $H$ with adjacent neighbors in $H$.
\label{lemma:9}
\end{lemma}

\begin{proof}
We may assume that $u$ and $v$ do not have a common neighbor in $H$. If $u$ (resp. $v$) has at least two neighbors in $H$, then let $u'$ and $u''$ (resp. $v'$ and $v''$) be the endpoints of the $u$-sector (resp. $v$-sector) that contains all neighbors of $v$ (resp. $u$) in $H$, and otherwise let $u' = u''$ (resp. $v' = v''$) be its unique neighbor in $H$. Without loss of generality, $u', v', v'', u''$ appear in this order along $H$. If $u$ and $v$ both have two non-adjacent neighbors in $H$ (and so by Lemma~\ref{lemma:1} and Lemma~\ref{lemma:2}, $u$ and $v$ each has at least three neighbors in $H$) and $uv \in E(G)$, then $H \cup \{u, v\}$ is a $uv$-turtle, a contradiction. If $u$ and $v$ both have two non-adjacent neighbors in $H$ and $uv \not \in E(G)$, then $G$ contains a theta between $u$ and $v$ through the paths $u \dd u' \dd H \setminus \{u'', v''\} \dd v' \dd v$, $u \dd u'' \dd H \setminus \{u', v'\} \dd v'' \dd v$, and $u \dd P \dd v$, a contradiction. If $u$ has two non-adjacent neighbors in $H$ and $v$ is a cap, then $G$ contains a pyramid from $u$ to $vv'v''$ through $u \dd P \dd v$, $u \dd u' \dd H \setminus \{u''\} \dd v'$, and $u \dd u'' \dd H \setminus \{u'\} \dd v''$, a contradiction. If $u$ has two non-adjacent neighbors in $H$ and $v$ is a pendant, then $G$ contains a theta between $u$ and $v'$ through $u \dd u' \dd H \setminus \{u''\} \dd v$, $u \dd P \dd v \dd v'$, and $u \dd u'' \dd H \setminus \{u'\} \dd v'$, a contradiction. Thus, neither $u$ nor $v$ has two non-adjacent neighbors in $H$. 

If $u$ and $v$ are both caps, then $H \cup P$ is a prism between $uu'u''$ and $vv'v''$, a contradiction. If $u$ is a cap and $v$ is a pendant, then $H \cup P$ is a pyramid from $v'$ to $uu'u''$. If $u$ and $v$ are both pendants, then $H \cup P$ is a theta between $u'$ and $v'$, unless $u'v'$ is an edge.
\end{proof}

Let $G \in \C$, let $H$ be a hole in $G$, and let $w$ be a major vertex for $H$. A path $N \subseteq H$ is an \emph{extended neighborhood} of $w$ in $H$ if there exists a $w$-sector $Q = x \hdots y$ such that $N = Q \cup (\{x', y'\} \cap N_H(w))$, where $x'$ and $y'$ are the neighbors of $x$ and $y$ in $H \setminus Q$, respectively. Two vertices $a, b \in H$ are \emph{distant in $H$ with respect to $w$} if $a, b$ are not contained in an extended neighborhood of $w$ in $H$. Note that if a vertex $v \in H$ is not adjacent to $w$, then $v$ is in exactly one extended neighborhood of $w$ in $H$.

Suppose $H$ is a hole in a graph $G \in \C$ and $u$ is a major vertex for $H$. The vertex $u$ is called a \emph{hub} if $N_H(u) = \{x, y, z, w\}$ where the vertices $x, y, z, w$ appear in that order in $H$, $xy, zw \in E(G)$, and $xw, zy \not \in E(G)$.

\begin{lemma}
Let $G \in \C$, let $H$ be a hole in $G$ of length greater than six, and let $w$ be a major vertex for $H$. Let $p \in V(G) \setminus V(H)$ be such that $pw \not \in E(G)$. Then, either $N_H(p)$ is contained in an extended neighborhood of $w$ in $H$ or $H \cup \{p, w\}$ is MNC configuration (6).
\label{lem:extended_nhbrhd}
\end{lemma}

\begin{proof}
If $p$ and $w$ are nested, then $N_H(p)$ is contained in an extended neighborhood of $w$ in $H$, so we may assume that $p$ and $w$ cross. It follows that $p$ is either a clone or a major vertex for $H$. If $p$ is a clone, then it follows from Lemma \ref{lemma:8} that $N_H(p)$ is contained in an extended neighborhood of $w$. Now, suppose $p$ is major. By Lemma \ref{lemma:7}, it follows that $H \cup \{w, p\}$ is MNC configuration (4), (5), (6), (7), or (8), and therefore either $N_H(p)$ is contained in an extended neighborhood of $w$ in $H$, or $H \cup \{w, p\}$ is MNC configuration (6).
\end{proof}

Let $G \in \C$, let $H$ be a hole in $G$, and let $w$ be a major vertex for $H$. We say that a path $P = p_1 \dd \hdots \dd p_k$ is \emph{$(H, w)$-significant} if there exist $a, b \in V(H)$ such that $a \in N_H(p_1) \setminus N_H(w)$, $b \in N_H(p_k)$, and $a$ and $b$ are distant in $H$ with respect to $w$. 

\begin{lemma}
Let $G \in \C$, let $H$ be a hole in $G$ of length greater than six, and let $w$ be a major vertex for $H$ such that either $w$ is not a hub or every major vertex for $H$ is a hub. Let $P = p_1 \dd \hdots \dd p_k$ be $(H, w)$-significant with $a, b \in V(H)$ as in the definition of a significant path. Let $Q = x \hdots y$ be the $w$-sector containing $a$. Suppose $w$ is anticomplete to $\{p_1, p_k\}$ and $p_1, p_k \not \in H$. Then, $p_k$ is anticomplete to $Q^*$.
\label{lemma:pk_anticomplete_Q*}
\end{lemma}
\begin{proof}
Since $p_1$ has a neighbor in the interior of a $w$-sector, $H \cup \{w, p_1\}$ is not MNC configuration~(6). By Lemma \ref{lem:extended_nhbrhd}, it follows that $N_H(p_1)$ is contained in an extended neighborhood of $w$. Note that, by definition, a $w$-sector of length greater than one is contained in exactly one extended neighborhood of $w$. Let $\overline{Q}$ be the extended neighborhood of $w$ containing $Q$, and suppose for sake of contradiction that $p_k$ has a neighbor in $Q^*$. Since $p_k$ has a neighbor in the interior of a $w$-sector, $H \cup \{w, p_k\}$ is not MNC configuration (6), so by Lemma \ref{lem:extended_nhbrhd}, $N_H(p_k)$ is contained in an extended neighborhood of $w$. Because $p_1$ has a neighbor $a$ in $Q^*$, it follows that $N_H(p_1) \subseteq \overline{Q}$. Similarly, because $p_k$ has a neighbor in $Q^*$, it follows that $N_H(p_k) \subseteq \overline{Q}$. Then, $a$ and $b$ are contained in an extended neighborhood of $w$, a contradiction.
\end{proof}

\begin{lemma}
Let $G \in \C$, let $H$ be a hole in $G$ of length greater than six, and let $w$ be a major vertex for $H$ such that either $w$ is not a hub or every major vertex for $H$ is a hub. Let $P = p_1 \dd \hdots \dd p_k$ be $(H, w)$-significant with $a, b \in V(H)$ as in the definition of a significant path. Assume that $P^*$ is anticomplete to $H$, $w$ is anticomplete to $\{p_1, p_k\}$, and $p_1, p_k \not \in H$. Then, $p_1$ and $p_k$ have a common neighbor in $H$. 
\label{lemma:common_nbr}
\end{lemma}

\begin{proof}
Assume for a contradiction that $p_1$ and $p_k$ do not have a common neighbor in $H$. In particular, $p_1 \neq p_k$. Suppose that $p_1$ and $p_k$ are nested. Then, they are strictly nested. By Lemma \ref{lemma:9}, $p_1$ and $p_k$ are pendants of $H$ with adjacent neighbors in $H$. It follows that $a$ and $b$ are adjacent, contradicting that $a$ and $b$ are distant in $H$ with respect to $w$. Therefore, $p_1$ and $p_k$ cross, and hence they are clone or major for $H$. Since $p_1$ and $p_k$ cross and they do not have a common neighbor in $H$, it follows that $p_1$ and $p_k$ are both major. If $w$ and $p_1$ are nested, then by Lemma \ref{lemma:pk_anticomplete_Q*}, $p_1$ and $p_k$ are nested. Hence, $w$ and $p_1$ cross. By Lemma \ref{lemma:8}, $H$ contains a $\{w, p_1\}$-complete edge. Let $Q = x \hdots y$ be the $w$-sector containing $a$. Let $x'$ and $y'$ be the neighbors of $x$ and $y$ in $H \setminus Q$, respectively. Since $p_1$ has a neighbor in the interior of a $w$-sector, $H \cup \{w, p_1\}$ is not MNC configuration (6). By Lemma \ref{lem:extended_nhbrhd}, it follows that $N_H(p_1) \subseteq Q \cup (\{x', y'\} \cap N_H(w))$, and by Lemma \ref{lemma:pk_anticomplete_Q*}, it follows that $N_H(p_k) \subseteq H \setminus Q^*$. Up to symmetry, suppose $\{w, p_1\}$ is complete to $\{x, x'\}$. Since $p_1$ and $p_k$ cross and $p_1$ and $p_k$ have no common neighbor in $H$, $p_1y', p_ky \in E(G)$ and $p_1y, p_ky' \not \in E(G)$. Also, $p_k$ has another neighbor in $H \setminus (Q \cup \{x', y'\})$ and in particular, $x'y' \not \in E(G)$. Because $p_1y' \in E(G)$, it follows that $wy' \in E(G)$. Let $p'$ be the neighbor of $p_1$ in $Q^*$ closest to $y$. Then, $G$ contains a pyramid from $p_1$ to $wyy'$ through $p_1 \dd y'$, $p_1 \dd x' \dd w$, and $p_1 \dd p' \dd Q \dd y$, a contradiction.
\end{proof}

We can now prove the main result of this section.

\begin{theorem}
Let $G \in \C$, let $H$ be a hole in $G$ of length greater than six, and let $w$ be a major vertex for $H$ such that either $w$ is not a hub or every major vertex for $H$ is a hub. Let $P = p_1 \dd \hdots \dd p_k$ be $(H, w)$-significant. Then, $w$ has a neighbor in $P$.
\label{theorem:star_cutset}
\end{theorem}

\begin{proof}
We may assume that no subpath of $P$ is $(H, w)$-significant. Suppose that $w$ is anticomplete to $P$. Let $a, b \in V(H)$ be as in the definition of a significant path, i.e., $a \in N_H(p_1) \setminus N_H(w)$, $b \in N_H(p_k)$, and $a$ and $b$ are distant in $H$ with respect to $w$. Let $Q = x \hdots y$ be the $w$-sector containing $a$, and let $x'$ and $y'$ be the neighbors of $x$ and $y$ in $H \setminus Q$, respectively. Possibly $x' = b$ (resp. $y' = b$), in which case $w$ is not adjacent to $x'$ (resp. $y'$) since $a$ and $b$ are distant in $H$ with respect to $w$.

Suppose that $k=1$. Since $a$ and $b$ are distant in $H$, $p_1 \notin V(H)$. So, by Lemma~\ref{lemma:pk_anticomplete_Q*} and Lemma~\ref{lemma:2}, $p_1$ is major but not a hub. It follows that $w$ is not a hub. But this contradicts Lemma~\ref{lem:extended_nhbrhd}. Therefore, $k > 1$.

\medskip

\noindent \emph{(1) $P$ is disjoint from $H$}.

Suppose that $p_i \in P$ is in $H$. Since no subpath of $P$ is $(H, w)$-significant, it follows that $a$ and $p_i$ are not distant in $H$ with respect to $w$. Then, either $p_i \in Q^*$ or $p_iw \in E(G)$. Since $w$ is anticomplete to $P$, $p_iw \not \in E(G)$, so $p_i \in Q^*$. Note that since $p_i \in Q^*$ and $b \notin Q$, $i < n$. Then, $p_{i+1} \dd \hdots \dd p_k$ is $(H, w)$-significant, a contradiction. This proves (1).

\medskip

\noindent \emph{(2) Either $N_H(P^*) \subseteq \{x\} \cup (N_H(w) \cap \{x'\})$ or $N_H(P^*) \subseteq \{y\} \cup (N_H(w) \cap \{y'\})$}. 

No vertex $p_i \in P \setminus \{p_1\}$ can have a neighbor in $Q^*$, otherwise $p_i \dd \hdots \dd p_k$ is $(H, w)$-significant. Similarly, no vertex $p_i \in P \setminus \{p_k\}$ can be adjacent to a vertex $v$ such that $a$ and $v$ are distant in $H$ with respect to $w$, otherwise $p_1 \dd \hdots \dd p_i$ is $(H, w)$-significant. It follows that $N_H(P^*) \subseteq \{x, y\} \cup (\{x', y'\} \cap N_H(w))$. 

Consider a vertex $p_i \in P^*$ such that $p_i$ has neighbors in $H$. Suppose $p_i$ is not a cap or a pendant, so by Lemmas~\ref{lemma:1} and~\ref{lemma:2}, $N_H(p_i) = \{x, y, x', y'\}$. Then, since $Q = x \hdots a \hdots y$ is a $w$-sector and $p_ix, p_iy \in E(G)$, by Lemma~\ref{lemma:4} it follows that $xy \in E(G)$, a contradiction. Therefore, $p_i$ is a cap or a pendant for every $p_i \in P^*$ that has a neighbor in $H$.

Now, assume that there exist $p_i, p_j \in P^*$ such that $N_H(p_i) \subseteq \{x\} \cup (N_H(w) \cap \{x'\})$ and $N_H(p_j) \subseteq \{y\} \cup (N_H(w) \cap \{y'\})$. Consider the shortest path $R$ from $\{x\} \cup (N_H(w) \cap \{x'\})$ to $\{y\} \cup (N_H(w) \cap \{y'\})$ through $P^*$. By Lemma \ref{lemma:9}, the endpoints of $R^*$ have a common neighbor in $H$ or are pendants of $H$ with adjacent neighbors, a contradiction since $b \in H \setminus (Q \cup (\{x', y'\} \cap N_H(w)))$. This proves (2).

\medskip

In view of (2), we assume from now on that $N_H(P^*) \subseteq \{x\} \cup (N_H(w) \cap \{x'\})$. Let $H_x$ and $H_y$ be the paths in $H$ from $a$ to $b$ through $x$ and $y$, respectively. Since $p_1$ has a neighbor in the interior of a $w$-sector, $H \cup \{w, p_1\}$ is not MNC configuration (6). By Lemma \ref{lem:extended_nhbrhd}, it follows that $N_H(p_1) \subseteq  Q \cup (\{x', y'\} \cap N_H(w))$, and by Lemma \ref{lemma:pk_anticomplete_Q*}, it follows that $N_H(p_k) \subseteq H \setminus Q^*$. Therefore, if $p_1$ and $p_k$ have a common neighbor $v$ in $H_y$, then $v \in \{y, y'\}$. Let $r_1 = r_s = v$ if $p_1$ and $p_k$ have a common neighbor $v$ in $H_y$. Otherwise, let $r_1$ be the neighbor of $p_1$ in $H_y$ that is furthest from $a$, and let $r_s$ be the neighbor of $p_k$ in $H_y$ that is furthest from $b$. Let $R = r_1 \hdots r_s$ be the path from $r_1$ to $r_s$ through $H_y$. Note that $r_1$ is between $a$ and $r_s$ unless $r_1 = y'$ and $r_s = y$, in which case $p_ky', p_1y \not \in E(G)$. It follows that $P \cup R$ is a hole when $P$ has length at least two.

\medskip

\noindent \emph{(3) $w$ has a neighbor in $R \cap \{y, y'\}$.}

Because $N_H(p_1) \subseteq Q \cup (\{x', y'\} \cap N_H(w))$, it follows that $N_{H_y}(p_1) \subseteq Q \cup (\{y'\} \cap N_H(w))$. If $r_1 \neq y'$, then $y \in R$ and $wy \in E(G)$, so $w$ has a neighbor in $R \cap \{y, y'\}$. If $r_1 = y'$, then $p_1y' \in E(G)$, so $wy' \in E(G)$, and $w$ has a neighbor in $R \cap \{y, y'\}$. This proves (3).

\medskip

\noindent \emph{(4) If $P \cup R$ is a hole, then $x$ is anticomplete to $P \cup R$.}

Let $J$ be the hole given by $P \cup R$. We prove a number of subclaims.

\medskip

\noindent \emph{(4.1) $x' \not \in J$}.

Suppose $x' \in J$.  Then $x'=r_s=b$, and so $x'$ is non-adjacent to $w$, since otherwise $a$ and $b$ are not distant in $H$ with respect to $w$. By Lemma~\ref{lemma:9} applied to $J$ and the path $x \dd w$, it follows that $w$ and $x$ are strictly nested with respect to the hole $J$, and so $x$ and $w$ are both pendants of $J$ with adjacent neighbors in $J$. Let $N_J(w)=\{x''\}$, then $x''$ is the neighbor of $x'$ in $H \setminus x$. Since $w$ has a unique neighbor in $J$, it follows from (3) that $x'' \in \{y,y'\}$. If $x'' = y'$, then $H = y \dd Q \dd x \dd x' \dd x'' \dd y$ and $N_H(w) = \{x, y, y'\}$, so $H \cup \{w\}$ is a pyramid, a contradiction. So $x'' = y$. But now $H = y \dd Q \dd x \dd x' \dd y$ and $N_H(y) = \{x, y\}$, so $H \cup \{w\}$ is a theta, a contradiction. This proves (4.1).

\medskip

\noindent \emph{(4.2) If $w$ is a hub, then $p_1$ is anticomplete to $\{x',y'\}$}.

Suppose $w$ is a hub and  $p_1$ is adjacent to $y'$ (the argument is similar if $p_1$ is adjacent to $x'$). Since $p_1$ is adjacent to $a$ and to $y'$, we deduce that $p_1$ is either a clone of $y$ or $p_1$ is major for $H$. Since $w$ is a hub, it follows that if $p_1$ is major for $H$, then $p_1$ is a hub for $H$. In both cases, $N_H(p_1) \setminus Q=\{y'\}$. Then, $G$ contains a pyramid from $y'$ to $xx'w$ through the paths $y' \dd w$, $y' \dd p_1 \dd Q \dd x$, and the path from $y'$ to $x'$ with interior in $H \setminus Q$, a contradiction. This proves (4.2).

\medskip

Suppose that $x$  has a neighbor in $J$. Since, by (4.1), $x' \not \in J$, it follows that $x$ is anticomplete to $R$, so $x$ has a neighbor in $P$. We apply Lemma~\ref{lemma:9} to $J$ and the path $x \dd w$. Since $w$ is anticomplete to $P$, we have that $x$ and $w$ have no common neighbor in $J$.  Consequently, $w$ and $x$ are strictly nested with respect to $J$, and so $x$ and $w$ are both pendants of $J$ with adjacent neighbors in $J$. Since $x$ is anticomplete to $R$, and $w$ is anticomplete to $P$, there are two possibilities:
\begin{enumerate}
\itemsep0em
\item  $N_J(x) = \{p_k\}$, $N_J(w)=\{r_s\}$, or
\item $N_J(x) = \{p_1\}$, $N_J(w)=\{r_1\}$.
\end{enumerate}
Suppose the former holds. By~(3), $r_s \subseteq \{y,y'\} \cap N(w)$. If $p_k$ is adjacent to $y$ (so $r_s = y$), then $G$ contains a theta between $x$ and $r_s$ given by the  paths $x \dd w \dd r_s$, $x \dd p_k \dd r_s$ and $x \dd Q \dd r_s$, a contradiction. It follows that $p_k$ is non-adjacent to $y$, and so $r_s=y'$. Then, $G$ contains a pyramid from $x$ to  $yy'w$ through the paths $x \dd w$, $x \dd p_k \dd r_s$ and $x \dd Q \dd y$, a contradiction. This proves that the former case does not hold, and therefore the latter holds.

By~(3), $r_1 \in \{y,y'\}$. If $r_1=y$, let $r_1'=y'$, and if $r_1=y'$ let $r_1'$ be the neighbor of $r_1$ in $H \setminus \{y\}$. Let $M$ be the subpath of $H \setminus \{a\}$ from $r_1'$ to $x'$. Suppose that both $w$ and $p_k$ have neighbors in $M^*$. Then there is a path $M'$ from $w$ to $p_k$ with $M'^* \subseteq M^*$. Now, $G$ contains a theta between $p_1$ and $w$ through the paths $p_1 \dd r_1 \dd w$,  $p_1 \dd x \dd w$ and $p_1 \dd P \dd p_k \dd M' \dd w$, a contradiction. This proves that either $p_k$ or $w$ is anticomplete to $M^*$.

\medskip

\noindent \emph{(4.3) $w$ has no neighbor in $M^*$}.

Suppose that $w$ has a neighbor in $M^*$. Then, $p_k$ is anticomplete to $M^*$. Since $w$ has a neighbor in $M^*$, and $N_R(w)=\{r_1\}$, it follows that $r_s \neq x'$ and $r_s$ is non-adjacent to $x'$. Consequently, $r_s \in \{y,y',r_1'\}$. Since $r_1 \in \{y,y'\} \cap N(w)$ and $a$ and $b$ are distant in $H$ with respect to $w$, it follows that either $b=r_1'$, or $b=x'$ and $x'$ is non-adjacent to $w$. Also, since $w$ has a unique neighbor in $J$, it holds that if $r_s=y$ then $r_1=y$. 

Suppose $x'$ has a neighbor in $P \setminus p_1$. Then, there is a path $P'$ from $x'$ to $p_k$ with interior in $P^*$. It follows from the minimality of $k$ that $P'$ is not $(H,w)$-significant. If $x'$ is non-adjacent to $w$, then $x'$ and $r_s$ are distant in $H$ with respect to $w$ since $w$ has a neighbor in $M^*$, and so $P'$ is $(H,w)$-significant, a contradiction. It follows that $x'$ is adjacent to $w$, and so $x' \neq b$ and $b = r_1'$. Suppose $r_1' \in R$. Then, $w$ is non-adjacent to $r_1'$ since $N_R(w) = \{r_1\}$. Now, we get a contradiction applying Lemma~\ref{lemma:9} to the path $x' \dd w$ and the hole $J$. This implies that $r_1' \not \in R$, and so $r_1=r_s$. (Indeed, if $r_1 \neq r_s$, then $r_1=y$, $r_s=y'$, $p_ky \notin E(G)$, $wy' \notin E(G)$, $r_1'=b=y'$, and hence the path $x' \dd w$ and the hole $J$ contradict Lemma~\ref{lemma:9}.) Since $J$ is a hole, we have $k>2$. Again, by Lemma~\ref{lemma:9} applied to the path $x' \dd w$ and the hole $J$, it follows that $x'$ is a pendant for $J$, and $N_J(x')=p_k$. But now $G$ contains a pyramid from $r_1$ to $xx'w$ with paths $r_1 \dd p_1 \dd x$, $r_1 \dd w$ and $r_1 \dd p_k \dd x'$, a contradiction. This proves that $x'$ is anticomplete to $P \setminus p_1$. 

Since $N_H(P^*) \subseteq \{x\} \cup (N_H(w) \cap \{x'\})$, $x'$ is anticomplete to $P \setminus p_1$, and $p_k$ is anticomplete to $M^*$, it follows that $P \setminus p_1$ is anticomplete to $M \setminus r_1'$. Since $x'$ is non-adjacent to $p_k$, it follows that $b=r_1'$. Since $k>1$ and by minimality of $k$, we deduce that $p_1$ is non-adjacent to $r_1'$. If $r_1' \in R$, then $G$ contains a theta between $p_1$ and $r_1'$ through the paths $p_1 \dd r_1 \dd r_1'$, $p_1 \dd x \dd H_x \dd r_1'$ (possibly shortcutting through the edge $p_1x'$), and $p_1 \dd P \dd p_k \dd r_1'$, a contradiction. This proves that $r_1' \not \in R$, and so $r_s \in \{y, r_1\}$. Since $N_J(x) = \{p_1\}$, $x'$ is anticomplete to $P \setminus \{p_1\}$, and $N_H(P^*) \subseteq \{x, x'\}$, it follows that $P^*$ is anticomplete to $H$. By Lemma~\ref{lemma:common_nbr}, $p_1$ and $p_k$ have a common neighbor in $H$. It follows that either $r_1=r_s=y$ or $r_1=r_s=y'$. Since $P \cup R$ is a hole, it follows that $k>2$. But now $G$ contains a pyramid from $p_1$ to $r_1r_1'p_k$ through the paths $p_1 \dd r_1$, $p_1 \dd P \dd p_k$, and $p_1 \dd x \dd H_x \dd r_1'$ (possibly shortcutting through $p_1x'$). This proves (4.3).

\medskip

\noindent \emph{(4.4) $w$ is a hub for $H$}.

If $w$ has no neighbor in $M \setminus \{x', y'\}$, then by Lemma \ref{lemma:1}, $w$ is a hub for $H$. Suppose $w$ has a neighbor in $M \setminus \{x', y'\}$. Since, by (4.3), $w$ is anticomplete to $M^*$, it follows that the only neighbor of $w$ in $M \setminus \{x', y'\}$ is $r_1'$. Thus, $r_1' \neq y'$, and so $r_1 = y'$. Since $r_1 = y'$ (and thus $p_1$ is adjacent to $y'$), it follows that $w$ is adjacent to $y'$. Since $N_J(w) = \{r_1\}$, it holds that $r_1 = r_s = y'$, and since $J$ is a hole, $P$ has length at least two. If $w$ is non-adjacent to $x'$, then $p_1$ is non-adjacent to $x'$ (as $N_H(p_1) \subseteq  Q \cup (\{x', y'\} \cap N_H(w))$) and hence $G$ contains a pyramid from $x$ to $y'r_1'w$ through $x \dd p_1 \dd y'$, $x \dd w$, and $x \dd x' \dd M^* \dd r_1'$, a contradiction. Therefore, $w$ is adjacent to $x'$. Suppose that the only neighbor of $p_k$ in $M$ is $r_1'$. Then, $G$ contains a pyramid from $p_1$ to $p_kr_1r_1'$ through $p_1 \dd r_1$, $p_1 \dd P \dd p_k$ (recall that $P$ has length at least two), and $p_1 \dd x \dd x' \dd H_x \dd r_1'$ (possibly shortcutting through the edge $p_1 \dd x'$), a contradiction. So $p_k$ has a neighbor in $M$ different from $r_1'$. Let $b'$ be the neighbor of $p_k$ in $M$ closest to $x'$. Now, $G$ contains a pyramid from $y'$ to $xx'w$ through $y' \dd w$, $y' \dd p_1 \dd x$, and $y' \dd p_k \dd b' \dd M \dd x'$, a contradiction. This proves (4.4).

\smallskip

It follows that $w$ is a hub and $N_H(w)=\{x,x',y,y'\}$. Consequently, by (4.2), $p_1$ is anticomplete to $\{x',y'\}$, and so $N_H(p_1) \subseteq Q$. Since $r_1 \in \{y,y'\}$ and $p_1$ is non-adjacent to $y'$, we deduce that $r_1 = y$. Since $N_J(w) = \{r_1\}$, it follows that $y' \not \in R$, and so $r_s = y$. Since $(H \setminus \{y'\}) \cup \{w,p_k\}$ is not a pyramid, it follows that $p_k$ is not a clone of $y'$. Since $a$ and $b$ are distant in $H$ with respect to $w$, it follows that $b \in H \setminus (Q \cup \{x', y'\})$. Since $p_k$ is adjacent to $y$ and to $b$, it holds that $p_k$ is a major vertex for $H$, and so $p_k$ is a hub by the assumption of the theorem. Consequently, $p_k$ is adjacent to $y'$, and $p_k$ is non-adjacent to $x$. Since $P \cup R$ is a hole and $r_1 = r_s$, it follows that $k>2$. But now $G$ contains a pyramid from $y$ to $wxx'$ given by paths $y \dd w$, $y \dd p_1 \dd x$ and $y \dd p_k \dd H_x \dd x'$.  This proves that $x$ is anticomplete to $J$ and completes the proof of (4).

\medskip

\noindent \emph{(5) If $x'$ is not anticomplete to $P \setminus p_k$, then $p_1x \in E(G)$.}

Assume $x'$ has a neighbor in $P \setminus p_k$, but $p_1x \not \in E(G)$. By our assumption, $N_H(P^*) \subseteq \{x\} \cup (N_H(w) \cap \{x'\})$, and by Lemma \ref{lem:extended_nhbrhd}, $N_H(p_1) \subseteq Q \cup (\{x', y'\} \cap N_H(w))$. Hence, $wx' \in E(G)$. Let $z$ be the neighbor of $x'$ in $P$ closest to $p_1$. Note that if $z \neq p_1$, then $P$ is of length at least two, so $P \cup R$ is a hole and by (4), $x$ is anticomplete to $P$. In particular, $x$ is anticomplete to $p_1 \dd P \dd z$. Consider the triangle given by $wxx'$. If $N_Q(p_1) = a$, then $G$ contains a pyramid from $a$ to $wxx'$ through $a \dd Q \dd y \dd w$, $a \dd Q \dd x$, and $a \dd P \dd z \dd x'$, a contradiction. Suppose $p_1$ has two non-adjacent neighbors in $Q$ and let $q$ and $q'$ be the neighbors of $p_1$ in $Q$ closest to $x$ and $y$, respectively. Then, $G$ contains a pyramid from $p_1$ to $wxx'$ through $p_1 \dd q' \dd Q \dd y \dd w$, $p_1 \dd q \dd Q \dd x$, and $p_1 \dd P \dd z \dd x'$, a contradiction. Finally, suppose $p_1$ has exactly two adjacent neighbors in $Q$ and let $N_H(p_1) = \{q, q'\}$, where $q$ is between $x$ and $q'$ in $Q$. Then, $G$ contains a prism between $p_1qq'$ and $x'xw$ through $p_1 \dd P \dd z \dd x'$, $q \dd Q \dd x$, and $q' \dd Q \dd y \dd w$, a contradiction. This proves (5).

\medskip

\noindent \emph{(6) If $P \cup R$ is a hole, then $\{x, x'\}$ is anticomplete to $P \setminus p_k$. In particular, $N_H(P^*) = \emptyset$.}

Suppose $P \cup R$ is a hole. By (4), $x$ is anticomplete to $P$. If $x'$ has neighbors in $P \setminus p_k$, then, by (5), $p_1x \in E(G)$, contradicting that $x$ is anticomplete to $P$. This proves the first assertion. Next, suppose that $N_H(P^*) \neq \emptyset$. Then, $P^* \neq \emptyset$, and so $P \cup R$ is a hole. Now, by the first assertion, $\{x, x'\}$ is anticomplete to $P \setminus p_k$. But $N_H(P^*) \subseteq \{x, x'\}$, a contradiction. This proves (6).

\medskip

By (6), $N_H(P^*) = \emptyset$, and so the symmetry between $x$ and $y$ is restored. Let $T = N_H(p_1) \cap N_H(p_k)$. By Lemma \ref{lemma:common_nbr}, $T \neq \emptyset$. Because $N_H(p_1) \subseteq Q \cup \{x', y'\}$ and $N_H(p_k) \subseteq H \setminus Q^*$, it follows that $T \subseteq \{x, x', y, y'\}$. Suppose first that one of $T \cap \{x, x'\}$ and $T \cap \{y, y'\}$ is empty. We may assume up to symmetry that $T \subseteq \{x, x'\}$. Because $p_1$ and $p_k$ do not have a common neighbor in $\{y, y'\}$, it follows that $P \cup R$ is a hole. Then, by (6), $\{x, x'\}$ is anticomplete to $p_1$, a contradiction. Therefore, we may assume that $T \cap \{x,x'\} \neq \emptyset$ and $T \cap \{y,y'\} \neq \emptyset$. By Lemma \ref{lemma:4} and since $p_k$ is anticomplete $Q^*$, it follows that $p_k$ is adjacent to at most one of $x$ and $y$, and so not both $x$ and $y$ are in $T$. Suppose that $x'$ and $y$ are both in $T$. Then, $w$ has three neighbors in the hole given by $x' \dd x \dd Q \dd y \dd p_k \dd x'$ and $w$ is not a clone or a major vertex for this hole, contradicting Lemmas~\ref{lemma:1} and~\ref{lemma:2}. This proves that not both $x'$ and $y$ are in $T$. By symmetry, not both $y$ and $x'$ are in $T$. It follows that $T=\{x',y'\}$. Because $p_1$ and $w$ are major and non-adjacent, and $p_1$ is adjacent to $x'$ and $y'$, it follows by Lemma \ref{lemma:7} that $H \cup \{p_1, w\}$ is MNC configuration (5). Therefore, $p_1$ is adjacent to $x$ and $y$. Since $T=\{x',y'\}$, it follows that $\{x,y\}$ is anticomplete to $V(P) \setminus \{p_1\}$. Further, because $p_1$ is not a hub, it follows that $w$ is not a hub, so $w$ has neighbors in $H \setminus (Q \cup \{x', y'\})$. Note also that $b \in H \setminus (Q \cup \{x', y'\})$.  Then, $G$ contains a theta between $p_1$ and $w$, through $x$, $y$, and $P \cup (H \setminus (Q \cup \{x', y'\}))$, a contradiction.
\end{proof}

Let $H = h_1 \dd h_2 \dd \dots \dd h_k \dd h_1$ be a hole in a graph $G \in \C$ and let $v \in V(G)$. We say that $v$ is a \emph{gem-center} if $k \geq 5$ and $N_H(v) = \{h_1, h_2, h_3, h_4\}$.

\begin{corollary}
Let $G \in \C$ and let $H$ be a hole in $G$ of length greater than six. Let $w$ be a major vertex for $H$ such that $w$ is not complete to $H$, $w$ is not a gem-center, and either $w$ is not a hub or every major vertex for $H$ is a hub. Then, $w$ is the center of a star cutset in $G$.
\end{corollary}
\begin{proof}
Let $u \in H$ such that $uw \not \in E(G)$. We claim that there exists a vertex $v \in H$ such that $u$ and $v$ are distant in $H$ with respect to $w$. Suppose otherwise. Let $Q$ be the $w$-sector containing $u$ and let $\overline{Q}$ be the extended neighborhood of $w$ containing $Q$. It follows that $\overline{Q} = H$, so $N_H(w)$ is contained in a subpath of $H$ of length at most three. Since $w$ is major, it follows that $N_H(w)$ is a subpath of $H$ of length exactly three, so $w$ is a gem-center, a contradiction. Let $v \in H$ be such that $u$ and $v$ are distant in $H$ with respect to $w$. It follows from Theorem \ref{theorem:star_cutset} that $w$ has a neighbor in the interior of every path from $u$ to $v$. Therefore, $u$ and $v$ are in different components of $G \setminus (N[w] \setminus v)$, so $w$ is the center of a star cutset in $G$. 
\end{proof}

\section{Structure of proper separators} \label{sec:structure_proper}

In this section, we consider minimal separators of graphs in $\C$. We start with the following result concerning minimal separators that are cliques.

\begin{lemma}[\hspace{1sp}\cite{BPS}]
For every graph $G$, there are at most $\mathcal{O}(|V(G)|)$ minimal clique separators of $G$ and they can be enumerated in time $\mathcal{O}(|V(G)||E(G)|)$.
\label{lem:clique_sep}
\end{lemma}

A separator in a graph is \emph{proper} if it is minimal and not a clique. By Lemma \ref{lem:clique_sep}, we restrict our attention here to proper separators. Our goal is to prove that graphs in $\C$ have polynomially many proper separators.

Let $C$ be a minimal separator of a graph $G$. A connected component $D$ of $G \setminus C$ is a \emph{full component for $C$} if every vertex of $C$ has a neighbor in $D$, i.e., $N(D) = C$. Recall that there are at least two full components for every minimal separator. The next lemma, while not necessary for our results, is a convenient observation about full components for proper separators of graphs in $\C$.

\begin{lemma}
If $C$ is a proper separator of a graph $G \in \C$, then there are exactly two full components for $C$.
\label{lemma:two_full_comps}
\end{lemma}

\begin{proof}
Let $c_1c_2$ be a non-edge in $C$, and suppose that there are three full components for $C$. Then, $G$ contains a path from $c_1$ to $c_2$ through each of the three full components, and so $G$ contains a theta between $c_1$ and $c_2$, a contradiction.
\end{proof}

For the rest of this section, we let $C$ be a proper separator of a graph $G \in \C$, and we denote by $L$ and $R$ the two full components for $C$. Let $H$ be a hole with $V(H) \cap V(C) = \{c_1, c_2\}$, and let $H_L$ and $H_R$ be the two paths of $H$ between $c_1$ and $c_2$. We say that $H$ is a \emph{$(C, c_1, c_2)$-hole} if $H_L^* \subseteq L$ and $H_R^* \subseteq R$.  A vertex $v \in V(G)$ is \emph{$(c_1, c_2)$-heavy} with respect to $H$ if $v$ is major for $H$ and $c_1, c_2$ are distant in $H$ with respect to $v$. Note that if $v \in V(G)$ is $(c_1, c_2)$-heavy with respect to $H$, then $v$ has a neighbor in $H_L^*$ and a neighbor in $H_R^*$, and therefore $v \in C$. The \emph{frame} of $H$ is given by $F(H) = (c_1, c_2, \ell_1', \ell_1, r_1, r_1', \ell_2', \ell_2, r_2, r_2')$, where $\ell_1$ is the neighbor of $c_1$ in $H_L$, $\ell_1'$ is the neighbor of $\ell_1$ in $H_L^*$ if $\ell_1 \neq \ell_2$, and otherwise $\ell_1' = \ell_1 = \ell_2$. We define similarly $\ell_2, \ell_2', r_1, r_1', r_2, r_2'$. We denote by $V(F)$ the vertices of $F$. We call $F$ a $(C, c_1, c_2)$-frame if $F$ is the frame of a $(C, c_1, c_2)$-hole. A hole $H$ is an \emph{$F$-hole} if $H$ is a $(C, c_1, c_2)$-hole with frame $F$.

\begin{lemma}
Let $H$ be a $(C, c_1, c_2)$-hole with frame $F = (c_1, c_2, \ell_1', \ell_1, r_1, r_1',$ $\ell_2', \ell_2, r_2, r_2')$. Assume that $v \in V(G) \setminus V(H)$ has a neighbor both in $H_L^* \setminus \{\ell_1, \ell_2\}$ and in $H_R^* \setminus \{r_1, r_2\}$. Then, $v$ is $(c_1, c_2)$-heavy with respect to $H$.
\label{lemma:heavy_nbrs_interior}
\end{lemma}
\begin{proof}
Suppose that $v$ is not $(c_1, c_2)$-heavy with respect to $H$. Then, $c_1$ and $c_2$ are in an extended neighborhood $\overline{Q}$ of $v$. Let $\overline{Q} = Q \cup (N_H(v) \cap \{x', y'\})$, where $Q = x \hdots y$ is a $v$-sector, and $x'$ and $y'$ are the neighbors of $x$ and $y$ in $H \setminus Q^*$, respectively. Then, $c_1$ and $c_2$ are either in $V(Q)$ or have a neighbor in $V(Q)$. Since $v$ has a neighbor in $H_L^* \setminus \{\ell_1, \ell_2\}$, it follows that $H_L^* \setminus V(Q) \neq \emptyset$. Similarly, $H_R^* \setminus V(Q) \neq \emptyset$. 

Suppose first that $c_1$ is not adjacent to $v$. Let $S$ be the $v$-sector of $H$ that contains $c_1$. Since $c_1v \notin E(G)$ and $c_1 \in \overline{Q}$, it follows that $S = Q$. Since $c_2 \in \overline{Q}$, either $v$ has no neighbor in $H_L^* \setminus \{\ell_1, \ell_2\}$ or $v$ has no neighbor in $H_R^* \setminus \{r_1, r_2\}$, a contradiction. Thus, $c_1v \in E(G)$, and similarly $c_2v \in E(G)$. But then, since $c_1, c_2 \in \overline{Q}$, either $H_L^* \setminus V(Q) = \emptyset$ or $H_R^* \setminus V(Q) = \emptyset$, a contradiction.
\end{proof}

The \emph{potential} of a $(C, c_1, c_2)$-hole $H$ is the total number of $(c_1, c_2)$-heavy vertices with respect to $H$. The following lemma shows that the potential of a $(C, c_1, c_2)$-hole only depends on its frame.

\begin{lemma}
Let $H_1$ and $H_2$ be $(C, c_1, c_2)$-holes with the same frame, given by $F(H_1) = F(H_2) = (c_1, c_2, \ell_1', \ell_1, r_1, r_1',$ $\ell_2', \ell_2, r_2, r_2')$. Then, $v \in V(G)$ is $(c_1, c_2)$-heavy with respect to $H_1$ if and only if $v$ is $(c_1, c_2)$-heavy with respect to $H_2$. In particular, the potential of $H_1$ and the potential of $H_2$ are equal.
\label{lemma:potential_depends_on_frame}
\end{lemma}

\begin{proof}
Suppose $v \in V(G)$ is $(c_1, c_2)$-heavy with respect to $H_1$ and not with respect to $H_2$.

\medskip

\noindent \emph{(1) If $v$ has no neighbor in $H_{2L}^* \setminus \{\ell_1, \ell_2\}$, then $N(v) \cap H_{1L}^* \subseteq \{\ell_1, \ell_2\}$. Similarly, if $v$ has no neighbor in $H_{2R}^* \setminus \{r_1, r_2\}$, then $N(v) \cap H_{1R}^* \subseteq \{r_1, r_2\}$.}

By symmetry, it suffices to prove the first statement. So assume that $v$ has no neighbor in $H_{2L}^* \setminus \{\ell_1, \ell_2\}$. We may assume that $\ell_1, \ell_1', \ell_2, \ell_2'$ are all distinct, since otherwise the result clearly holds. In particular, $H_1$ and $H_2$ are both of length greater than six.

Since $v$ is $(c_1, c_2)$-heavy with respect to $H_1$, $v$ has a neighbor in both $H_{1L}^*$ and $H_{1R}^*$. Suppose $v$ is anticomplete to $\{\ell_1, \ell_2\}$. Then, there exists a path $P = p_1 \dd \dots \dd p_k$ in $(H_1 \setminus \{\ell_1,\ell_2\}) \cup \{v\}$ such that $P \cap H_2 = \emptyset$, $p_1$ has a neighbor in $H_{2L}^*$, $p_k$ has a neighbor in $H_{2R}$, and $P^*$ is anticomplete to $H_2$. Note that $v \in P$ and $p_1\in L$ (i.e. $p_1 \neq v$), so $P$ is of length at least 1. But then $P$ and $H_2$ contradict Lemma \ref{lemma:9}. So $v$ is not anticomplete to $\{\ell_1, \ell_2\}$. Thus, we may assume that $v$ is adjacent to $\ell_1$. Let $Q$ be the $v$-sector of $H_1$ that contains $\ell_1'$. Then, $\ell_1 \in V(Q)$. Since $c_1$ and $c_2$ are distant in $H_1$ with respect to $v$, it follows that $v$ is a major vertex for $H_1$. We claim that $v$ is not a hub for $H_1$. Suppose $v$ is a hub for $H_1$. Since $\ell_1 \in N(v)$ and $\ell_1' \not \in N(v)$, it follows that $c_1 \in N(v)$. But then $c_1$ and $c_2$ are not distant in $H_1$ with respect to $v$, a contradiction. This proves that $v$ is not a hub for $H_1$. Now, since $v$ is anticomplete to $\ell_1' \dd H_{2L} \dd \ell_2'$, it follows from Theorem \ref{theorem:star_cutset} that $\ell_1'$ and $\ell_2'$ are not distant in $H_1$ with respect to $v$. Since $v$ is not adjacent to $\ell_2'$, it follows that $\ell_2' \in Q$, so $N(v) \cap H_{1L}^* \subseteq \{\ell_1, \ell_2\}$. This proves (1).

\medskip

By Lemma \ref{lemma:heavy_nbrs_interior}, we may assume that $v$ has no neighbor in $H_{2L}^* \setminus \{\ell_1, \ell_2\}$. By (1), it follows that $N(v) \cap H_{1L}^* \subseteq \{\ell_1, \ell_2\}$.

\medskip

\noindent \emph{(2) $v$ has a neighbor in $H_{2R}^* \setminus \{r_1, r_2\}$.}

Assume that $v$ has no neighbor in $H_{2R}^* \setminus \{r_1, r_2\}$. Then, by (1), $N(v) \cap H_{1R}^* \subseteq \{r_1, r_2\}$. But now, $N(v) \cap H_1 = N(v) \cap H_2$, and so $c_1$ and $c_2$ are distant in $H_2$ with respect to $v$, a contradiction. This proves (2).

\medskip

Since $c_1$ and $c_2$ are not distant in $H_2$ with respect to $v$, there exists an extended neighborhood $\overline{Q}$ of $v$ in $H_2$ such that $c_1$ and $c_2$ are both in $\overline{Q}$. Let $\overline{Q} = Q \cup (N_{H_2}(v) \cap \{x', y'\})$ where $Q = x \hdots y$ is a $v$-sector in $H_2$ and $x'$ and $y'$ are the neighbors of $x$ and $y$ in $H_2 \setminus Q$, respectively. Since $\overline{Q}$ contains $c_1$ and $c_2$, it follows that either $H_{2L} \subseteq \overline{Q}$ or $H_{2R} \subseteq \overline{Q}$. Suppose that $H_{2L} \subseteq \overline{Q}$.  Since $c_1$ and $c_2$ are distant in $H_1$ with respect to $v$ and $N(v) \cap H_{1L}^* \subseteq \{\ell_1, \ell_2\}$, we may assume that $v$ is adjacent to $\ell_1$. Since $c_1$ is in $\overline{Q}$, it follows that $v$ is also adjacent to $c_1$. Because $\overline{Q}$ is an extended neighborhood of $v$ in $H_2$ that contains $c_1$ and $c_2$, $v$ is either non-adjacent to $\ell_2$, or $v$ is adjacent to $\ell_2$ and $c_2$. But now $c_1$ and $c_2$ are not distant in $H_1$ with respect to $v$, a contradiction.  Therefore, $H_{2R} \subseteq \overline{Q}$. However, by (2), $v$ has a neighbor in $H_{2R}^* \setminus \{r_1, r_2\}$, a contradiction.
\end{proof}

Let $F = (c_1, c_2, \ell_1', \ell_1, r_1, r_1',$ $\ell_2', \ell_2, r_2, r_2')$ be a $(C, c_1, c_2)$-frame. We say that a vertex $v \in V(G)$ is \emph{$F$-heavy} if there exists an $F$-hole $H$ such that $v$ is $(c_1, c_2)$-heavy with respect to $H$.  Note that Lemma \ref{lemma:potential_depends_on_frame} implies that an $F$-heavy vertex $v$ is $(c_1, c_2)$-heavy with respect to every hole $H$ with frame $F$. A vertex $v$ that is not $F$-heavy is said to be \emph{$F$-light}. The \emph{potential} of $F$ is the total number of $F$-heavy vertices. 

Let $c_1, c_2 \in C$. We denote by $\dist_L(c_1, c_2)$ and $\dist_R(c_1, c_2)$ the length of the shortest path from $c_1$ to $c_2$ through $L$ and $R$, respectively, and we let $\dist(c_1, c_2) = \min(\dist_R(c_1, c_2), \dist_L(c_1, c_2))$. We say that $(c_1, c_2)$ is a \emph{long pair} of $C$ if $\dist(c_1, c_2) \geq 4$. A $(C, c_1, c_2)$-frame $F$ is \emph{long} if $(c_1, c_2)$ is a long pair of $C$. A proper separator $C$ is \emph{rich} if there exist $c_1, c_2 \in C$ such that $(c_1, c_2)$ is a long pair, and \emph{poor} otherwise.

\begin{lemma}
Suppose $F$ is a $(C, c_1, c_2)$-frame, $H$ is an $F$-hole, and $c_3 \in C \setminus \{c_1, c_2\}$ is $F$-light. Let $P = p_k \dd \hdots \dd p_1 \dd c_3 \dd q_1 \dd \hdots \dd q_j$ be a path such that $c_3 \dd p_1 \dd \hdots \dd p_k$ is a path from $c_3$ to $H_L^*$ through $L$ and $c_3 \dd q_1 \dd \hdots \dd q_j$ is a path from $c_3$ to $H_R^*$ through $R$ (possibly $c_3 = p_k$ or $c_3 = q_j$), and assume $P$ has length at least two. Then, up to symmetry between $c_1$ and $c_2$, one of the following holds: 
\begin{enumerate}[(i)]
\itemsep-0.2em
\item $c_1$ and $c_2$ are anticomplete to $P^*$, $N_H(p_k) = \{\ell_1, c_1\}$, and $N_H(q_j) = \{c_1, r_1\}$,

\item $c_2$ is anticomplete to $P^*$, $c_1$ has neighbors in $P^*$, $p_k$ is either adjacent to $c_1$ or a pendant of $H$ with neighbor $\ell_1$, and $q_j$ is either adjacent to $c_1$ or a pendant of $H$ with neighbor $r_1$.
\end{enumerate}
\vspace{-\baselineskip}
\vspace{0.2cm}
\label{lemma:butterfly_end_shapes}
\end{lemma}

\begin{proof}
If both $c_1$ and $c_2$ have neighbors in $P^*$, then $G$ contains a theta between $c_1$ and $c_2$ through $H_L$, $H_R$, and $P^*$, so we may assume that $c_2$ is anticomplete to $P^*$.

Suppose $c_1$ is also anticomplete to $P^*$. By Lemma \ref{lemma:9}, either $p_k$ and $q_j$ have a common neighbor in $H$, or their neighbors in $H$ form an edge. Since $p_k$ has a neighbor in $H_L^*$ and $q_j$ has a neighbor in $H_R^*$, it follows that the neighbors of $p_k$ and $q_j$ in $H$ do not form an edge. Hence, we may assume that $p_k$ and $q_j$ are both adjacent to $c_1$. If $p_k$ and $q_j$ both have neighbors in $H \setminus \{c_1\}$ other than $\ell_1$ and $r_1$, respectively, then $G$ contains a theta between $p_k$ and $q_j$ through $P$, $c_1$, and $H \setminus \{\ell_1, c_1, r_1\}$, a contradiction. Suppose $N_H(q_j) = \{c_1, r_1\}$ and $p_k$ has a neighbor in $H_L$ other than $\ell_1$. Let $s$ be the neighbor of $p_k$ in $H_L$ closest to $c_2$. Then, $G$ contains a pyramid from $p_k$ to $q_jc_1r_1$ through $p_k \dd P \dd q_j$, $p_k \dd c_1$, and $p_k \dd s \dd H \setminus \{c_1\} \dd r_1$, a contradiction. By definition, $p_k$ and $q_j$ have neighbors in $H_L^*$ and $H_R^*$, respectively, and so $N_H(p_k) = \{\ell_1, c_1\}$ and $N_H(q_j) = \{c_1, r_1\}$, and outcome (i) holds. 

Next, suppose $c_1$ has neighbors in $P^*$. Let $r$ be the closest neighbor of $c_1$ to $p_k$ in $P^*$. By Lemma \ref{lemma:9} applied to the path $p_k \dd P \dd r$, either $p_k$ and $r$ have a common neighbor in $H$, or $p_k$ and $r$ are pendants of $H$ with adjacent neighbors in $H$. Since $N_H(r) = \{c_1\}$, either $p_k$ is adjacent to $c_1$ or $N_H(p_k) = \{\ell_1\}$. By symmetry, either $q_j$ is adjacent to $c_1$ or $N_H(q_j) = \{r_1\}$, and outcome (ii) holds.
\end{proof}

Let $H$ be an $F$-hole and let $c_3 \in C \setminus \{c_1, c_2\}$ be $F$-light. A \emph{$c_3$-butterfly} is a path $P = p_k \dd \hdots \dd p_1 \dd c_3 \dd$ $q_1 \dd \hdots \dd q_j$, where $c_3 \dd p_1 \dd \hdots \dd p_k$ is a shortest path from $c_3$ to $H_L^*$ through $L$ and $c_3 \dd q_1 \dd \hdots \dd q_j$ is a shortest path from $c_3$ to $H_R^*$ through $R$ (possibly $p_k = c_3$ or $c_3 = q_j$). We call the path $c_3 \dd p_1 \dd \hdots \dd p_k$ the \emph{left wing} of $P$, and the path $c_3 \dd q_1 \dd \hdots \dd q_j$ the \emph{right wing} of $P$. We say that $c_3$ is a \emph{central vertex} of $P$ if $c_3 \neq p_k, p_{k-1}, q_j, q_{j-1}$. 

The following results deal with the structure of $c_3$-butterflies. 

\begin{lemma}
Suppose $F$ is a $(C, c_1, c_2)$-frame, $H$ is an $F$-hole, and $c_3 \in C \setminus \{c_1, c_2\}$ is $F$-light. Suppose further that if $C$ is a rich separator, then $F$ is long, and if $C$ is a poor separator, then $\dist(c_1, c_2)$ is maximum over all non-adjacent pairs in $C$. Let $P$ be a $c_3$-butterfly and assume $c_2$ is anticomplete to $P^*$. Suppose that $c_3$ is a central vertex of $P$. Then, $(c_3, c_2)$ is a long pair of $C$. In particular, $C$ is a rich separator.
\label{lemma:c3c2_long_pair}
\end{lemma}

\begin{proof}
Assume for a contradiction that $(c_3, c_2)$ is not a long pair of $C$. Then, there exists a path from $c_3$ to $c_2$ of length less than or equal to three through $L$ or through $R$. First, assume that there exists a path of length two from $c_3$ to $c_2$, say $c_3 \dd x \dd c_2$, and without loss of generality let $x \in L$. Because $P$ is a butterfly and $c_3$ is a central vertex of $P$, neither $c_3$ nor $c_3 \dd x$ is the left wing of a $c_3$-butterfly, so $x \notin H$ and $x$ is anticomplete to $H_L^*$. It follows that $N_H(x) \subseteq \{c_1, c_2\}$. If $N_H(x) = \{c_1, c_2\}$, then $G$ contains a theta between $c_1$ and $c_2$ through $H_L$, $H_R$, and $x$, so $N_H(x) = \{c_2\}$. If $c_1$ has neighbors in $P^*$, then $G$ contains a theta between $c_1$ and $c_2$ through $H_L$, $H_R$, and $P^* \cup \{x\}$, so $c_1$ is anticomplete to $P^*$. It follows from Lemma \ref{lemma:butterfly_end_shapes} that $N_H(p_k) = \{\ell_1, c_1\}$. Now, $G$ contains a pyramid from $c_2$ to $p_k\ell_1c_1$ through $c_2 \dd x \dd c_3 \dd P \dd p_k$, $c_2 \dd H_L \dd \ell_1$, and $c_2 \dd H_R \dd c_1$, a contradiction. Therefore, there is no path of length two from $c_3$ to $c_2$. 

Next, let $c_3 \dd x \dd y \dd c_2$ be a path of length three from $c_3$ to $c_2$, and without loss of generality let $x, y \in L$. Since $\dist(c_3, c_2) = 3$, it follows that $\dist(c_1, c_2) \geq 3$. In particular, $c_1$ is not adjacent to $y$. Because $c_3$ is a central vertex of $P$, it follows that neither $c_3$ nor $c_3 \dd x$ is the left wing of a $c_3$-butterfly. Therefore, $x, y \notin H$ and $N_H(x) \subseteq \{c_1\}$. Suppose $x$ is adjacent to $c_1$. Then, $x$ and $y$ are strictly nested with respect to $H$. By Lemma \ref{lemma:9}, $x$ and $y$ are pendants of $H$ with adjacent neighbors in $H$, so $c_1$ is adjacent to $c_2$, a contradiction. Hence, $x$ is anticomplete to $H$.

Suppose first that $c_1$ is adjacent to $c_3$. Consider the path $y - x - c_3$. By Lemma \ref{lemma:9}, either $y$ and $c_3$ have a common neighbor in $H$, or $y$ and $c_3$ are pendants of $H$ with adjacent neighbors in $H$. Since $N_H(c_3) = \{c_1\}$ and $y$ is not adjacent to $c_1$, it follows that $y$ is a pendant with $N_H(y) = \{\ell_1\}$. But $c_2 \in N_H(y)$, a contradiction. This shows that $c_1$ is not adjacent to $c_3$. Next, suppose that $c_1$ has a neighbor in $\{q_1, q_2, \dots, q_j\}$. Let $t$ be minimum such that $c_1$ is adjacent to $q_t$. Let $Q$ be a path from $y$ to $q_t$ with $Q^* \subseteq \{x, c_3, q_1, \dots, q_{t-1}\}$. By Lemma \ref{lemma:9}, either $y$ and $q_t$ have a common neighbor in $H$, or $y$ and $q_t$ are pendants of $H$ with adjacent neighbors in $H$. Suppose $t \neq j$, so $N_H(q_t) = \{c_1\}$. Since $y$ is not adjacent to $c_1$, it follows that $y$ is a pendant of $H$ and $N_H(y) = \{\ell_1\}$. But $c_2 \in N_H(y)$, a contradiction. Therefore, $t = j$. Since $q_j$ is adjacent to $c_1$ and $q_j$ has a neighbor in $H_R^*$, $q_j$ is not a pendant of $H$. Therefore, $q_j$ and $y$ have a common neighbor in $H$. Since $y \in L$ and $q_j \in R$, the common neighbor of $y$ and $q_j$ is $c_2$. Then, $q_j$ is a common neighbor of $c_1$ and $c_2$, contradicting that $\dist(c_1, c_2) \geq 3$. This proves that $c_1$ is anticomplete to $\{c_3, q_1, \dots, q_j\}$. 

Since $q_j$ is not adjacent to $c_1$, by Lemma \ref{lemma:butterfly_end_shapes}, $N_H(q_j) = \{r_1\}$. Now, consider the path $Q = y \dd x \dd c_3 \dd \hdots \dd q_j$. By Lemma \ref{lemma:9}, either $y$ and $q_j$ are pendants of $H$ with adjacent neighbors in $H$, or $y$ and $q_j$ have a common neighbor in $H$. Since $y \in L$ and $N_H(q_j) = \{r_1\}$, $y$ and $q_j$ do not have a common neighbor in $H$. Therefore, $r_1$ is adjacent to $c_2$, contradicting that $\dist(c_1, c_2) \geq 3$.
\end{proof}

By Lemma \ref{lemma:butterfly_end_shapes} and Lemma \ref{lemma:c3c2_long_pair}, if $C$ is a poor separator and $\dist(c_1, c_2)$ is maximum over all non-adjacent pairs in $C$, then $c_3$ is not a central vertex of $P$. The following two lemmas prove a similar result for rich separators.

\begin{lemma} 
Suppose $C$ is a rich separator, $(c_1, c_2)$ is a long pair of $C$, $F$ is a $(C, c_1, c_2)$-frame, $H$ is an $F$-hole, $c_3 \in C \setminus \{c_1, c_2\}$ is $F$-light, and $P = p_k \dd \hdots \dd c_3 \dd \hdots \dd q_j$ is a $c_3$-butterfly with $c_2$ anticomplete to $P^*$. Assume that $c_3$ is a central vertex of $P$ and let $w$ be an $F$-heavy vertex. Let $s$ be the neighbor of $p_k$ in $H_L$ closest to $c_2$, and let $t$ be the neighbor of $q_j$ in $H_R$ closest to $c_2$. Let $S$ be the path from $s$ to $t$ in $H \setminus \{c_1\}$, and consider the $(C, c_3, c_2)$-hole given by $J = P \cup S$. Then, $w$ is a $(c_3, c_2)$-heavy vertex with respect to $J$.
\label{lemma:w_stays_heavy}
\end{lemma}

\begin{proof}
By Lemma \ref{lemma:butterfly_end_shapes}, $p_k$ is either adjacent to $c_1$ or $N_H(p_k) = \{\ell_1\}$. Since $(c_1, c_2)$ is a long pair of $C$, it follows that $s \neq \ell_2$. By symmetry, $t \neq r_2$. Assume for a contradiction that $w$ is not $(c_3, c_2)$-heavy with respect to $J$.

\medskip

\noindent \emph{(1) If $p_k$ is adjacent to $c_1$, then $w$ has a neighbor in $p_k \dd s \dd H_L \dd \ell_2$.}

Because $(c_1, c_2)$ is a long pair of $C$, $H$ has length at least eight. By Lemma \ref{lemma:potential_depends_on_frame}, $w$ is $(c_1,c_2)$-heavy with respect to $H$. Suppose $w$ is anticomplete to the path given by $p_k \dd s \dd H_L \dd \ell_2$. If $w$ is not complete to $\{c_1, c_2\}$, then since $c_1$ and $c_2$ are distant in $H$ with respect to $w$, the path $p_k \dd s \dd H_L \dd \ell_2$ is $(H, w)$-significant and $w$ is not a hub for $H$, a contradiction to Theorem \ref{theorem:star_cutset}. So $w$ is complete to $\{c_1, c_2\}$. Let $w'$ be the neighbor of $w$ in $H_L^*$ that is closest to $\ell_2$. Note that $w'$ exists and $w' \neq \ell_1$ since $w$ is $(c_1, c_2)$-heavy with respect to $H$. Also, observe that there is a path $Q$ from $w$ to $p_k$ through $P \cup H_R \setminus \{c_1, r_1\}$: either $w$ has a neighbor in $P \setminus \{p_k\}$, or there is a path $Q' = q_j \dd t \dd H_R \dd w'' \dd w$, where $w''$ is the neighbor of $w$ in $t \dd H_R \dd c_2$ closest to $t$ (possibly $w'' = c_2)$. Let $s'$ be the neighbor of $p_k$ in $w' \dd H_L \dd c_2$ closest to $w'$; note that $s'$ exists since $s$ is in $w' \dd H_L \dd c_2$ (possibly $s' = w'$). Now, $G$ contains a theta between $p_k$ and $w$ through $p_k \dd s' \dd H_L \dd w' \dd w$, $p_k \dd c_1 \dd w$, and $p_k \dd Q \dd w$, a contradiction. This proves (1).

\medskip

\noindent \emph{(2) $w$ has a neighbor in $J_L^* \setminus N_J(c_3)$ and a neighbor in $J_R^* \setminus N_J(c_3)$.}

By symmetry, it suffices to show that $w$ has a neighbor in $J_L^* \setminus N_J(c_3)$. Assume first that $s = \ell_1$. Then, $H_L^* \subseteq J_L^*$. Further, $c_3$ does not have a neighbor in $H_L^*$, otherwise $c_3 = p_k$, contradicting that $c_3$ is a central vertex of $P$. Finally, since $w$ is $(c_1, c_2)$-heavy with respect to $H$, $w$ has a neighbor in $H_L^*$, and we are done. Hence, we may assume that $s \neq \ell_1$. By Lemma \ref{lemma:butterfly_end_shapes}, $p_k$ is adjacent to $c_1$, so by (1), $w$ has a neighbor in $p_k \dd s \dd H_L \dd \ell_2$. Note that $p_k \dd s \dd H_L \dd \ell_2 \subseteq J_L^*$, so $w$ has a neighbor in $J_L^*$. Further, $c_3$ does not have a neighbor in $p_k \dd s \dd H_L \dd \ell_2$, otherwise $c_3$ is adjacent to $p_k$, contradicting that $c_3$ is a central vertex of $P$. Therefore, $w$ has a neighbor in $J_L^* \setminus N_J(c_3)$. This proves (2).

\medskip

By Lemma \ref{lemma:heavy_nbrs_interior}, $w$ does not have a neighbor in both $J_L^* \setminus N_J(\{c_3, c_2\})$ and $J_R^* \setminus N_J(\{c_3, c_2\})$, so by (2) we may assume that $\ell_2$ is the only neighbor of $w$ in $J_L^* \setminus N_J(c_3)$. By (2) $w$ has a neighbor in $J_R^* \setminus N_J(c_3)$ and since $c_3$ and $c_2$ are not distant in $J$ with respect to $w$, it follows that $w$ is also adjacent to $c_2$. Because $c_1$ and $c_2$ are distant in $H$ with respect to $w$, it follows that $s \neq \ell_1$ and $w$ has neighbors in the interior of $c_1 \dd H_L \dd s$. Since $s \neq \ell_1$, by Lemma \ref{lemma:butterfly_end_shapes}, it follows that $p_k$ is adjacent to $c_1$. Thus, $w$ and $p_k$ cross with respect to $H$. There are two cases: either $p_k$ is a clone of $\ell_1$, or $p_k$ is major and by Lemma \ref{lemma:7} $H \cup \{w, p_k\}$ is MNC configuration (4), (5), (6), (7), or (8). Suppose the first case holds, so $p_k$ is a clone of $\ell_1$. Since $p_k$ is a clone of $\ell_1$, it holds that $s = \ell_1'$. By Lemma~\ref{lemma:8}, $H$ contains a $\{w, p_k\}$-complete edge, so $w$ is adjacent to $\ell_1$ and $c_1$ (note that $w$ is not adjacent to $s = \ell_1'$ since $\ell_2$ is the only neighbor of $w$ in $J_L^* \setminus N_J(c_3)$). Then, $N(w) \cap H_L = \{c_1, \ell_1, \ell_2, c_2\}$, so $c_1$ and $c_2$ are not distant in $H$ with respect to $w$, a contradiction. Therefore, $p_k$ is major and $H \cup \{w, p_k\}$ is MNC configuration (4), (5), (6), (7), or (8). It follows that there is a $\{w, p_k\}$-complete edge in $c_1 \dd H_L \dd s$. Let $e = v_1v_2$ be a $\{w, p_k\}$-complete edge in $c_1 \dd H_L \dd s$ such that $v_2$ is between $v_1$ and $s$ in $H_L$. Suppose that $w$ is not a cap with respect to $J$ and let $u$ be the neighbor of $w$ in the $p_kc_2$-subpath of $J \setminus \ell_2$ that is closest to $p_k$. Then $G$ contains a theta between $p_k$ and $w$ through $p_k \dd v_1 \dd w$, $p_k \dd s \dd H_L \dd \ell_2 \dd w$, and $p_k \dd J \setminus \{\ell_2\} \dd u \dd w$. Hence, $N_J(u) = \{\ell_2, c_2\}$. Then, $G$ contains a pyramid from $p_k$ to $w\ell_2 c_2$ through $p_k \dd v_1 \dd w$, $p_k \dd s \dd H_L \dd \ell_2$, and $p_k \dd P \dd q_j \dd t \dd H_R \dd c_2$, a contradiction.
\end{proof}

\begin{lemma}
Suppose $C$ is a rich separator, $F$ is a $(C, c_1, c_2)$-frame with maximum potential over all long frames, and $H$ is an $F$-hole. Let $c_3 \in C \setminus \{c_1, c_2\}$ be $F$-light, and let $P = p_k \dd \hdots \dd c_3 \dd \hdots \dd q_j$ be a $c_3$-butterfly. Then, $c_3$ is not a central vertex of $P$.
\label{lemma:rich_not_int}
\end{lemma}

\begin{proof}
Suppose that $c_3$ is a central vertex of $P$. By Lemma \ref{lemma:butterfly_end_shapes}, we may assume that $c_2$ is anticomplete to $P^*$. Furthermore, by Lemma \ref{lemma:butterfly_end_shapes}, $p_k$ is adjacent to $c_1$ or $N_H(p_k) = \{\ell_1\}$. Since $(c_1, c_2)$ is a long pair of $C$, it follows that $p_k$ is anticomplete to $\{c_2, \ell_2\}$. By symmetry, $q_j$ is anticomplete to $\{c_2, r_2\}$. Let $r$ be the neighbor of $p_k$ in $H_L$ closest to $c_2$, and let $s$ be the neighbor of $q_j$ in $H_R$ closest to $c_2$. Let $S$ be the path in $H$ from $r$ to $s$ through $c_2$. Let $J$ be the $(C,c_3,c_2)$-hole given by $P \cup S$, and let $F'$ be the $(C,c_3,c_2)$-frame of $J$. By Lemma \ref{lemma:w_stays_heavy} it follows that every $F$-heavy vertex is $(c_2,c_3)$-heavy with respect to $J$, and therefore is $F'$-heavy. Now, consider $c_1$ with respect to the hole $J$. Either $c_1$ is adjacent to $p_k$, or $\ell_1$ is in $J$ and $c_1$ is adjacent to $\ell_1$. Similarly, either $c_1$ is adjacent to $q_j$, or $r_1$ is in $J$ and $c_1$ is adjacent to $r_1$. Then, $c_1$ has a neighbor in $J_L^* \setminus N_J(\{c_3, c_2\})$ and a neighbor in $J_R^* \setminus N_J(\{c_3, c_2\})$, so by Lemma \ref{lemma:heavy_nbrs_interior}, $c_1$ is a $(c_3, c_2)$-heavy vertex of $J$. Finally, it follows from Lemma \ref{lemma:c3c2_long_pair} that $(c_3, c_2)$ is a long pair of $C$. Then, $F'$ is a long $(C, c_3, c_2)$-frame with higher potential than $F$, a contradiction. 
\end{proof}

We call a $(C, c_1, c_2)$-frame $F$ \emph{optimal} if one of the following holds:
\begin{enumerate}[(i)]
\itemsep-0.2em
\item $C$ is a rich separator and $F$ has maximum potential over all long frames of $C$
\item $C$ is a poor separator and $\dist(c_1, c_2)$ is maximum over all non-adjacent pairs of vertices in $C$.
\end{enumerate}

The following theorem combines the results of Lemma \ref{lemma:butterfly_end_shapes}, Lemma \ref{lemma:c3c2_long_pair}, and Lemma \ref{lemma:rich_not_int}. 

\begin{theorem}
Let $F$ be an optimal $(C, c_1, c_2)$-frame, $H$ be an $F$-hole, $c_3 \in C \setminus \{c_1, c_2\}$ be $F$-light, and $P = p_k \dd \hdots \dd c_3 \dd \hdots \dd q_j$ be a $c_3$-butterfly. Then, $c_3$ is not a central vertex of $P$.
\label{theorem:4_final_thm}
\end{theorem}

\section{Constructing proper separators} \label{sec:constructing_separators}

In this section, we show how to use the structure results from previous sections to prove the main result of the paper. Our goal is to reconstruct proper separators $C$ given only an optimal frame $F$ of $C$, and two 4-tuples $M_1(C), M_2(C)$ of vertices in $C$. We first show that we can construct an $F$-hole $H$, and then show that we can construct three sets $C_1, C_2, C_3$ such that $C = C_1 \cup C_2 \cup C_3$.

We begin with a key observation about the structure of graphs in $\C$.

\begin{lemma}
If $G \in \C$, then $G$ does not contain a $3$-creature.
\label{lemma:three_paths_lemma}
\end{lemma}

\begin{proof}
Assume that $G$ contains a $3$-creature with notation as in the definition of a $k$-creature. Suppose first that $x_3$ is adjacent to $x_1$ and $x_2$. Let $Q_A$ be a path from $x_1$ to $x_2$ through $A$, and let $Q_B$ be a path from $y_1$ to $y_2$ through $B$. Then, $x_1 \dd Q_A \dd x_2 \dd y_2 \dd Q_B \dd y_1 \dd x_1$ is a hole $H$ in $G$. Let $R_B = y_3 \dd \hdots \dd b$ be a path from $y_3$ to $Q_B$ through $B$. Consider the path $R = x_3 \dd y_3 \dd R_B \dd b$. Since $x_3$ and $b$ are strictly nested with respect to $H$, by Lemma \ref{lemma:9}, it follows that $x_3$ and $b$ are pendants of $H$ with adjacent neighbors in $H$. However, $x_3$ is adjacent to $x_1$ and $x_2$, a contradiction. 

We may therefore assume that $x_1$ is not adjacent to $x_2$ and $y_1$ is not adjacent to $y_2$. Let $Q_A$ be a path from $x_1$ to $x_2$ through $A$, and let $Q_B$ be a path from $y_1$ to $y_2$ through $B$. Then, $x_1 \dd Q_A \dd x_2 \dd y_2 \dd Q_B \dd y_1 \dd x_1$ is a hole $H$ in $G$. Let $R_A = x_3 \dd \hdots \dd a$ be a path from $x_3$ to $Q_A^*$ through $A$ and let $R_B = y_3 \dd \hdots \dd b$ be a path from $y_3$ to $Q_B^*$ through $B$. Consider the path $R = a \dd R_A \dd x_3 \dd y_3 \dd R_B \dd b$. If $\{x_1, y_1, x_2, y_2\}$ is anticomplete to $R^*$, then $a$ and $b$ are strictly nested with respect to $H$, and $a$ and $b$ are not pendants of $H$ with adjacent neighbors in $H$, contradicting Lemma \ref{lemma:9}. Hence, one of $x_1, y_1, x_2, y_2$ has a neighbor in $R^*$. In particular, $R^*$ is not empty. Suppose $x_1$ and $x_2$ both have neighbors in $R^*$. Then, $G$ contains a theta between $x_1$ and $x_2$ through $Q_A$, $Q_B$, and $R^*$, a contradiction. Therefore, not both $x_1$ and $x_2$ have neighbors in $R^*$. Similarly, not both $y_1$ and $y_2$ have neighbors in $R^*$. Since $\{x_1, y_1, x_2, y_2\}$ is not anticomplete to $R^*$, we may assume that $x_1$ has a neighbor in $R^*$. If $y_2$ also has a neighbor in $R^*$, then $G$ contains a theta between $x_1$ and $y_2$ through $Q_A$, $Q_B$, and $R^*$, a contradiction. Therefore, $y_2$ is anticomplete to $R^*$. 

Let $c$ be the closest neighbor of $x_1$ to $x_3$ in $R_A$. Suppose $y_1$ is anticomplete to $R^*$ and consider the path $c \dd R_A \dd x_3 \dd y_3 \dd R_B \dd b$. Then, $c$ and $b$ are strictly nested with respect to $H$. Since $b$ has a neighbor in $Q_B^*$, $c$ and $b$ are not pendants of $H$ with adjacent neighbors in $H$, contradicting Lemma \ref{lemma:9}. Hence, $y_1$ has a neighbor in $R^*$. Let $a'$ be the neighbor of $a$ in $Q_A$ closest to $x_2$, and let $b'$ be the neighbor of $b$ in $Q_B$ closest to $y_2$. Let $H'$ be the hole given by $H' = x_3 \dd R_A \dd a \dd a' \dd Q_A \dd x_2 \dd y_2 \dd Q_B \dd b' \dd b \dd R_B \dd y_3 \dd x_3$. Since $x_1$ and $y_1$ are strictly nested with respect to $H'$, by Lemma \ref{lemma:9}, $x_1$ and $y_1$ are pendants of $H'$ with adjacent neighbors in $H'$. Therefore, $N_{H'}(x_1) = \{x_3\}$ and $N_{H'}(y_1) = \{y_3\}$. In particular, $a'x_1, b'y_1 \notin E(G)$. Since $R^* \neq \emptyset$, without loss of generality $y_3 \neq b$. Now, $G$ contains a theta between $x_3$ and $y_1$ through $y_3$, $x_1$, and $x_3 \dd R_A \dd a \dd a' \dd Q_A \dd x_2 \dd y_2 \dd Q_B \dd y_1$, a contradiction.
\end{proof}

For the rest of the section, unless otherwise specified, let $C$ be a proper separator of $G \in \C$ and let $F = (c_1, c_2, \ell_1', \ell_1, r_1, r_1',$ $\ell_2', \ell_2, r_2, r_2')$ be an optimal $(C, c_1, c_2)$-frame. We denote by $G_F$ the graph $G \setminus (N(\{c_1, c_2, \ell_1, r_1, \ell_2, r_2\}) \setminus \{\ell_1', \ell_2', r_1', r_2'\})$. The following two lemmas show that we can construct a set $W = W(F)$ containing every $F$-heavy vertex $v$ such that $v \in V(G_F)$.

\begin{lemma}
Let $H$ be an $F$-hole and let $v \in C \cap V(G_F)$ be major for $H$. Then, $v$ is $F$-heavy. 
\label{lemma:major_is_heavy}
\end{lemma}

\begin{proof}
Assume that $v$ is not $F$-heavy, and let $P$ be a $v$-butterfly.  Since $v$ is major for $H$, $v$ must be an endpoint of $P$. Since $v \in V(G_F)$, $v$ is anticomplete to $\{c_1, c_2, \ell_1, r_1, \ell_2, r_2\}$, and so by Lemma~\ref{lemma:butterfly_end_shapes}, $P$ is of length at most one. Since $v$ is $F$-light, by Lemma \ref{lemma:heavy_nbrs_interior}, $P$ is of length exactly one. But then since $v$ is anticomplete to $\{c_1, c_2, \ell_1, r_1, \ell_2, r_2\}$, $P$ and $H$ contradict Lemma \ref{lemma:9}.
\end{proof}

We call $v \in C$ a \emph{$(c_1, c_2)$-strong} vertex of $G$ if $c_1$ and $c_2$ belong to different components of $G \setminus N[v]$. Note that given a graph $G$, and $v, c_1, c_2 \in V(G)$, one can determine if $v$ is $(c_1, c_2)$-strong in time $\mathcal{O}(|V(G)|^2)$.

\begin{lemma}
One can construct in polynomial time a set $W = W(F)$ that contains all $F$-heavy vertices $v$ such that $v$ is anticomplete to $\{c_1, c_2, \ell_1, \ell_2, r_1, r_2\}$ and $W \subseteq C$.
\label{lemma:construct_Y1}
\end{lemma}

\begin{proof}
Let $H$ be an $F$-hole where the path from $\ell_1'$ to $\ell_2'$ through $H_L$ is a shortest path from $\ell_1'$ to $\ell_2'$ in $L$, and the path from $r_1'$ to $r_2'$ through $H_R$ is a shortest path from $r_1'$ to $r_2'$ through $R$. We may assume that $H$ has length greater than six since otherwise $W$ is empty. Let $X_1$ be the set of all $(c_1, c_2)$-strong vertices of $G_F$, and let $X_2$ be the set of all $(c_1, c_2)$-strong vertices of $G_F \setminus X_1$. Note that $X_1$ and $X_2$ can be constructed in time $\mathcal{O}(|V(G)|^3)$. If $v$ is $(c_1, c_2)$-strong, then $v$ has a neighbor in $H_L^*$ and a neighbor in $H_R^*$, so $v \in C$. It follows that $X_1 \cup X_2 \subseteq C$. We claim that $W = X_1 \cup X_2$ contains all $F$-heavy vertices $v$ such that $v$ is anticomplete to $\{c_1, c_2, \ell_1, r_1, \ell_2, r_2\}$. 

By Theorem \ref{theorem:star_cutset}, $X_1$ contains all $F$-heavy vertices $v$ in $G_F$ such that $v$ is not a hub of $H$. Now, consider $G_F \setminus X_1$. Every $F$-heavy vertex in $G_F \setminus X_1$ is a hub. Suppose $v \in V(G_F \setminus X_1)$ is a major vertex for $H$ and $v$ is $F$-light. By Lemma \ref{lemma:major_is_heavy}, $v \in L$ or $v \in R$. Without loss of generality suppose $v \in L$. Since $v$ is a major vertex for $H$ and $v \in V(G_F \setminus X_1)$, it follows that $N(v) \cap (H_L^* \setminus \{\ell_1, \ell_2\})$ is not contained in a path of length three, so there exists a shorter path from $\ell_1'$ to $\ell_2'$ in $L$ through $v$, a contradiction. Therefore, every major vertex for $H$ in $G_F \setminus X_1$ is $F$-heavy, so every major vertex for $H$ in $G_F \setminus X_1$ is a hub. Then, it follows from Theorem \ref{theorem:star_cutset} that $X_2$ contains every $F$-heavy vertex of $H$ in $G_F \setminus X_1$.

Finally, let $v$ be an $F$-heavy vertex in $G$ such that $v$ is anticomplete to $\{c_1, c_2, \ell_1, r_1, \ell_2, r_2\}$.  If $v$ is not a hub, then $v$ is an $F$-heavy vertex in $G_F$, so $v \in X_1$. If $v$ is a hub and $v \not \in X_1$, then $v$ is an $F$-heavy vertex of $G_F \setminus X_1$,  so $v \in X_2$. 
\end{proof}

\begin{lemma}
Given an optimal frame $F$ of $C$, one can construct in polynomial time an $F$-hole $H$.
\label{lemma:construct_H}
\end{lemma}
\begin{proof}
By Lemma \ref{lemma:construct_Y1}, we can construct the set $W = W(F) \subseteq C$ of all $F$-heavy vertices $v$ such that $v$ is anticomplete to $\{c_1, c_2, \ell_1, \ell_2, r_1, r_2\}$. Let $H$ be the graph given by the union of $V(F)$, a shortest path $Q_L$ from $\ell_1'$ to $\ell_2'$ through $G_F \setminus W$, and a shortest path $Q_R$ from $r_1'$ to $r_2'$ through $G_F \setminus W$. We claim that $H$ is an $F$-hole.

If $Q_L \subseteq L$ and $Q_R \subseteq R$, then clearly $H$ is an $F$-hole, so assume without loss of generality that $Q_L \not \subseteq L$. Let $\ell^*$ be the vertex of $Q_L \setminus L$ closest to $\ell_1'$ on $Q_L$. Since $\ell^*$ has a neighbor in $L$ and $\ell^* \not \in L$, it follows that $\ell^* \in C$. Suppose $\ell^*$ is $F$-heavy. Since $W$ contains all $F$-heavy vertices anticomplete to $\{c_1, c_2, \ell_1, \ell_2, r_1, r_2\}$, it follows that $\ell^*$ has a neighbor in $\{c_1, c_2, \ell_1, \ell_2, r_1, r_2\}$, a contradiction. Therefore, $\ell^*$ is $F$-light. Let $J$ be an $F$-hole. Let $P_R$ be a path from $\ell^*$ to $J_R^*$ through $R$, and let $P_L$ be a path from $\ell^*$ to $J_L^*$ contained in  $\ell^* \dd Q_L \dd \ell_1'$. Consider the path $P = P_L \dd \ell^* \dd P_R$ and let $p_k$ be the end of $P_L$ with neighbors in $J_L^*$. Suppose $P$ is of length at least two. By Lemma \ref{lemma:butterfly_end_shapes}, it follows that either $p_k$ is adjacent to $c_1$ or $p_k$ is a pendant with $N_{J}(p_k) = \{\ell_1\}$. Since $P_L \subseteq V(G_F \setminus W)$, $p_k$ is not adjacent to $c_1$ or $\ell_1$, a contradiction. Therefore, $P$ is of length at most one. Because $\ell^*$ is $F$-light and $\ell^*$ is anticomplete to $\{c_1, c_2, \ell_1, \ell_2, r_1, r_2\}$, it follows that $\ell^*$ does not have a neighbor in both $J_L^*$ and $J_R^*$. Therefore, $P$ has length exactly one, and $P$ and $J$ contradict Lemma \ref{lemma:9}.
\end{proof}

By Lemma \ref{lemma:construct_H}, we can construct an $F$-hole $H$. Let $c_3 \in C$ be $F$-light, and let $P = p_k \dd \hdots \dd p_1 \dd c_3 \dd $ $q_1 \dd \hdots \dd q_j$ be a $c_3$-butterfly for $H$. By Theorem \ref{theorem:4_final_thm}, $c_3$ is not a central vertex of $P$. We call $c_3$ an \emph{$L$-end} vertex if $c_3 = p_k$, and an \emph{$L$-adjacent} vertex if $c_3 = p_{k-1}$. We define similarly \emph{$R$-end} and \emph{$R$-adjacent}. The following lemma shows that every $L$-adjacent vertex is in the neighborhood of two vertices in $L$ and that every $R$-adjacent vertex is in the neighborhood of two vertices in $R$.

\begin{lemma}
Let $X \subseteq N(H_L^*) \cap L$ be a minimal subset of $N(H_L^*) \cap L$ such that every $L$-adjacent vertex has a neighbor in $X$. Then, $|X| \leq 2$. Similarly, let $Y \subseteq N(H_R^*) \cap R$ be a minimal subset of $N(H_R^*) \cap R$ such that every $R$-adjacent vertex has a neighbor in $Y$. Then, $|Y| \leq 2$.
\label{lemma:first_catch}
\end{lemma}

\begin{proof}
Suppose $|X| > 2$ and let $x_1, x_2, x_3 \in X$. It follows from the minimality of $X$ that for every $x_i \in X$ there exists $y_i \in C$ such that $y_i$ is $L$-adjacent and $N_X(y_i) = \{x_i\}$. For $i = 1, 2, 3$, let $P_i$ be the right wing of a $y_i$-butterfly. Let $A = H_L^*$ and let $B = (P_1 \setminus \{y_1\}) \cup (P_2 \setminus \{y_2\}) \cup (P_3 \setminus \{y_3\}) \cup H_R^*$. Then, $A$ is anticomplete to $B$, $G[A]$ and $G[B]$ are connected, and for $i = 1, 2, 3$, $x_i$ has a neighbor in $A$ and is anticomplete to $B$, and $y_i$ has a neighbor in $B$ and is anticomplete to $A$. It follows that $A \cup B \cup \{x_1, x_2, x_3\} \cup \{y_1, y_2, y_3\}$ is a 3-creature, contradicting Lemma \ref{lemma:three_paths_lemma}.
\end{proof}

Let $X = \{x_1, x_2\}$ and $Y = \{y_1, y_2\}$ be as in Lemma \ref{lemma:first_catch} (so possibly $x_1 = x_2$ or $y_1 = y_2$). Let $M_1(C) = (x_1, x_2, y_1, y_2)$. Let $C_L = N(H_L^* \cup \{x_1, x_2\})$ and $C_R = N(H_R^* \cup \{y_1, y_2\})$. Note that $C_L$ and $C_R$ depend only on $H$ and $M_1(C)$.

\begin{lemma}
$C_L \cap C_R \subseteq C \subseteq C_L \cup C_R$.
\label{lem:CL_cap_CR_in_C}
\end{lemma}

\begin{proof}
It follows from the definition of $C_L$ that $C_L \subseteq L \cup C$. Similarly, $C_R \subseteq R \cup C$. Since $L$ and $R$ are disjoint, it follows that $C_L \cap C_R \subseteq C$. Next, suppose $c \in C$. Since $\{c_1, c_2\} \subseteq C_L \cap C_R$, we may assume that $c \in C \setminus \{c_1, c_2\}$. If $c$ is $F$-heavy, then by Lemma \ref{lemma:potential_depends_on_frame} $c$ is $(c_1, c_2)$-heavy with respect to $H$, and hence $c$ has neighbors in both $H_L^*$ and $H_R^*$, so $c \in C_L \cap C_R$. Therefore, we may assume that $c$ is $F$-light. By Theorem \ref{theorem:4_final_thm}, $c$ is either $L$-end, $L$-adjacent, $R$-end, or $R$-adjacent. If $c$ is $L$-end, then $c$ has a neighbor in $H_L^*$. If $c$ is $L$-adjacent, it follows from Lemma \ref{lemma:first_catch} that $c$ has a neighbor in $\{x_1, x_2\}$. Therefore, if $c$ is $L$-end or $L$-adjacent, then $c \in C_L$. By symmetry, if $c$ is $R$-end or $R$-adjacent, $c \in C_R$. Hence, $C \subseteq C_L \cup C_R$.
\end{proof}

Let $C_1 = C_L \cap C_R$. Note that for every $s \in C$, if there exists an $s$-butterfly $P$ of length zero or one, then $s \in C_1$. Let $D = V(G) \setminus (H \cup C_L \cup C_R)$. The following lemmas show how to identify the vertices of $C \setminus C_1$.

\begin{lemma}
Let $S \subseteq C_R \cap R$ be a minimal subset of $C_R \cap R$ such that for every vertex $z \in (C_L \setminus C_R) \cap C$, there exists a path from $z$ to $S$ through $D$. Then, $|S| \leq 2$. Similarly, let $T \subseteq C_L \cap L$ be a minimal subset of $C_L \cap L$ such that for every vertex $z \in (C_R \setminus C_L) \cap C$, there exists a path from $z$ to $T$ through $D$. Then, $|T| \leq 2$.
\label{lem:shortest_paths}
\end{lemma}

\begin{proof}
First, note that for every vertex $z \in (C_L \setminus C_R) \cap C$ there exists a path from $z$ to $C_R \cap R$ through $D$ given by a subpath of the right wing of a $z$-butterfly. Suppose $|S| > 2$ and let $s_1, s_2, s_3 \in S$. By the minimality of $S$, it follows that there exist $z_1, z_2, z_3 \in (C_L \setminus C_R) \cap C$ such that there exists a path $P_i$ from $z_i$ to $s_i$ through $D$ for $i = 1,2,3$, and there does not exist a path from $z_i$ to $s_j$ through $D$ for $1 \leq i \neq j \leq 3$. Let $z_1', z_2', z_3'$ be the neighbors of $z_1, z_2, z_3$ in $P_1, P_2, P_3$, respectively. Let $A = H_L^* \cup \{x_1, x_2\}$ and let $B = (P_1 \setminus \{z_1, z_1'\}) \cup (P_2 \setminus \{z_2, z_2'\}) \cup (P_3 \setminus \{z_3, z_3'\}) \cup (H_R^* \cup \{y_1, y_2\})$. Then, $A$ is anticomplete to $B$, $G[A]$ and $G[B]$ are connected, and for $i = 1, 2, 3$ $z_i$ has a neighbor in $A$ and is anticomplete to $B$, and $z_i'$ has a neighbor in $B$ and is anticomplete to $A$. It follows that $A \cup B \cup \{z_1, z_2, z_3\} \cup \{z_1', z_2', z_3'\}$ is a 3-creature, contradicting Lemma \ref{lemma:three_paths_lemma}.
\end{proof}

Let $S = \{r_a, r_b\}$ and $T = \{\ell_a, \ell_b\}$ be as in Lemma \ref{lem:shortest_paths} (possibly $r_a = r_b$ or $\ell_a = \ell_b$). Let $M_2(C) = (\ell_a, \ell_b, r_a, r_b)$. Let $C_2$ be the set of all vertices $c \in C_L$ such that there exists a path $P$ from $c$ to $\{r_a, r_b\}$ through $D$. Similarly, let $C_3$ be the set of all vertices $c\in C_R$ such that there exists a path $P$ from $c$ to $\{\ell_a, \ell_b\}$ through $D$. Note that $C_2$ and $C_3$ depend only on $H$, $W$, $C_L$, $C_R$, and $M_2(C)$.

\begin{lemma}
 $C_2 \cup C_3 \subseteq C$.
\label{lem:Y3_Y4_in_C}
\end{lemma}

\begin{proof}
Suppose $c \in C_L$ such that there exists a path $P$ from $c$ to $\{r_a, r_b\}$ through $D$. Since $c \in C_L$, $c$ has a neighbor in $L$, so some vertex of $P$ belongs to $C$. Since, by Lemma \ref{lem:CL_cap_CR_in_C}, $C \subseteq C_L \cup C_R$, no vertex of $P \setminus \{c\}$ is in $C$. It follows that $c \in C$. Therefore, $C_2 \subseteq C$. By symmetry, $C_3 \subseteq C$.
\end{proof}

\begin{lemma}
$C = C_1 \cup C_2 \cup C_3$. In particular, $C$ is uniquely determined by $F$, $M_1(C)$, and $M_2(C)$.
\end{lemma}

\begin{proof}By Lemmas \ref{lem:CL_cap_CR_in_C} and \ref{lem:Y3_Y4_in_C}, it follows that $C_1 \cup C_2 \cup C_3 \subseteq C$. Consider $c \in C$. We may assume $c \not \in C_1$. Then, by Lemma \ref{lem:CL_cap_CR_in_C}, either $c \in (C_L \setminus C_R) \cap C$ or $c \in (C_R \setminus C_L) \cap C$. If $c \in (C_L \setminus C_R) \cap C$, it follows from Lemma \ref{lem:shortest_paths} that there is a path $P$ from $c$ to $\{r_a, r_b\}$, so $c \in C_2$. Similarly, if $c \in (C_R \setminus C_L) \cap C$, then $c \in C_3$. Therefore, $C \subseteq C_1 \cup C_2 \cup C_3$.
\end{proof}

Let $C(F, M_1(C), M_2(C)) = C_1 \cup C_2 \cup C_3$ be the set constructed from $F$, $M_1(C)$, and $M_2(C)$, as described in this section. We proved that if $C$ is a proper separator and $F$ is an optimal frame of $C$, then $C(F, M_1(C), M_2(C)) = C$. The following corollary is a summary of the results presented in Section \ref{sec:constructing_separators} so far.

\begin{corollary}
Given the tuples $F$, $M_1$, and $M_2$, one can construct $C(F, M_1, M_2)$ in polynomial time. Further, if $F$ is an optimal frame of $C$, then $C(F, M_1(C), M_2(C)) = C$.
\label{corollary:sec4_summary}
\end{corollary}

Finally, we prove Theorem \ref{thm:main_thm}, which we restate here for convenience. Recall that by \cite{BBC}, to construct a list of all minimal separators of a graph, it suffices to prove that it has polynomially many minimal separators. We prove that $\C$ has the polynomial separator property and provide in addition a polynomial-time algorithm to construct the minimal separators of graphs in $\C$, which follows naturally from the results in this section. 

\begin{theorem}
Let $G \in \C$. One can construct a set $\mathcal{S}$ of size at most $|V(G)|^{18}$ in polynomial time such that $\mathcal{S}$ is the set of all minimal separators of $G$.
\label{thm:main_thm2}
\end{theorem}

\begin{proof}
Let $\mathcal{S} = \{\}$. By Lemma \ref{lem:clique_sep}, we add to $\mathcal{S}$ all minimal clique separators of $G$. Next, we list the proper separators of $G$. Let $T = (c_1, c_2, \ell_1', \ell_1, r_1, r_1', \ell_2', \ell_2, r_2, r_2', x_1, x_2, y_1, y_2, \ell_a, \ell_b, r_a, r_b)$ be an 18-tuple consisting of vertices in $V(G)$. Let $F^T = (c_1, c_2, \ell_1', \ell_1, r_1, r_1', \ell_2', \ell_2, r_2, r_2')$, $M_1^T = (x_1, x_2, y_1, y_2)$, and $M_2^T = (\ell_a, \ell_b, r_a, r_b)$. For every 18-tuple $T$, let $C^T = C(F^T, M_1^T, M_2^T)$. By Corollary \ref{corollary:sec4_summary}, $C^T$ can be constructed in polynomial time. We can test in time $\mathcal{O}(|E(G)||V(G)|)$ whether $C^T$ is a minimal separator of $G$. We add $C^T$ to $\mathcal{S}$ if and only if $C^T$ is a minimal separator of $G$. Clearly, $\mathcal{S}$ has size at most $|V(G)|^{18}$ and can be constructed in polynomial time.

It remains to show that $\mathcal{S}$ contains every minimal separator of $G$. Let $C$ be a minimal separator of $G$. We may assume that $C$ is proper. Let $F$ be an optimal frame of $C$ and let $T$ be the 18-tuple given by the union of $F$, $M_1(C)$, and $M_2(C)$, in that order. It follows from Corollary \ref{corollary:sec4_summary} that $C^T = C$, so $C \in \mathcal{S}$. 
\end{proof}


\end{document}